\documentclass[a4paper]{article}

\usepackage[all]{xy}
\usepackage[leqno]{amsmath}
\usepackage{amssymb}
\usepackage[only,mapsfrom,heavycircles,lightning]{stmaryrd}
\usepackage{latexsym}
\usepackage{graphicx}
\usepackage{graphbox}
\usepackage{mathrsfs}   
\usepackage{titlesec} 
\usepackage{paralist}
\usepackage{tocloft}
\usepackage{xcolor}
\usepackage{upgreek}
\usepackage{lmodern}
\usepackage[T1]{fontenc}
\usepackage[utf8]{inputenc}
\usepackage{calc}
\usepackage[normalem]{ulem}

\usepackage[hidelinks=true,hypertexnames=false,breaklinks=true]{hyperref} 
\usepackage{comment}

\newlength{\theorempostskipamount}
\makeatletter

\newenvironment{theorem}[1][]
{\paragraph{Theorem} {\normalfont #1} \it}
{\vspace{\the\theorempostskipamount}}
\def\theorem{\@ifnextchar[{\@theoremopt}{\@theoremplain}}
\def\@theoremplain{\paragraph{Theorem} \it}
\def\@theoremopt[#1]{\paragraph{Theorem \normalfont #1}  \it}

\newenvironment{lemma}[1][]
{\paragraph{Lemma} {\normalfont #1} \it}
{\vspace{\the\theorempostskipamount}}
\def\lemma{\@ifnextchar[{\@lemmaopt}{\@lemmaplain}}
\def\@lemmaplain{\paragraph{Lemma} \it}
\def\@lemmaopt[#1]{\paragraph{Lemma \normalfont #1}  \it}

\newenvironment{proposition}[1][]
{\paragraph{Proposition} {\normalfont #1} \it}
{\vspace{\the\theorempostskipamount}}
\def\proposition{\@ifnextchar[{\@propositionopt}{\@propositionplain}}
\def\@propositionplain{\paragraph{Proposition} \it}
\def\@propositionopt[#1]{\paragraph{Proposition \normalfont #1}  \it}

\newenvironment{definition}[1][]
{\paragraph{Definition} {\normalfont #1} \it}
{\vspace{\the\theorempostskipamount}}
\def\definition{\@ifnextchar[{\@definitionopt}{\@definitionplain}}
\def\@definitionplain{\paragraph{Definition} \it}
\def\@definitionopt[#1]{\paragraph{Definition \normalfont #1}  \it}

\newenvironment{corollary}[1][]
{\paragraph{Corollary} {\normalfont #1} \it}
{\vspace{\the\theorempostskipamount}}
\def\corollary{\@ifnextchar[{\@corollaryopt}{\@corollaryplain}}
\def\@corollaryplain{\paragraph{Corollary} \it}
\def\@corollaryopt[#1]{\paragraph{Corollary \normalfont #1}  \it}

\newenvironment{question}[1][]
{\paragraph{Question} {\normalfont #1} \it}
{\vspace{\the\theorempostskipamount}}
\def\question{\@ifnextchar[{\@questionopt}{\@questionplain}}
\def\@questionplain{\paragraph{Question} \it}
\def\@questionopt[#1]{\paragraph{Question \normalfont #1}  \it}

\newenvironment{problem}[1][]
{\paragraph{Problem} {\normalfont #1} \it}
{\vspace{\the\problempostskipamount}}
\def\problem{\@ifnextchar[{\@problemopt}{\@problemplain}}
\def\@problemplain{\paragraph{Problem} \it}
\def\@problemopt[#1]{\paragraph{Problem \normalfont #1}  \it}

\newenvironment{conjecture}[1][]
{\paragraph{Conjecture} {\normalfont #1} \it}
{\vspace{\the\theorempostskipamount}}
\def\conjecture{\@ifnextchar[{\@conjectureopt}{\@conjectureplain}}
\def\@conjectureplain{\paragraph{Conjecture} \it}
\def\@conjectureopt[#1]{\paragraph{Conjecture \normalfont #1}  \it}

\newenvironment{remark}[1][]
{\paragraph{Remark} {\normalfont #1}}
{\vspace{\the\theorempostskipamount}}
\def\remark{\@ifnextchar[{\@remarkopt}{\@remarkplain}}
\def\@remarkplain{\paragraph{Remark}}
\def\@remarkopt[#1]{\paragraph{Remark \normalfont #1}}

\newenvironment{remarks}[1][]
{\paragraph{Remarks} {\normalfont #1}}
{\vspace{\the\theorempostskipamount}}
\def\remarks{\@ifnextchar[{\@remarksopt}{\@remarksplain}}
\def\@remarksplain{\paragraph{Remarks} \it}
\def\@remarksopt[#1]{\paragraph{Remarks \normalfont #1}  \it}

\newenvironment{example}[1][]
{\paragraph{Example} {\normalfont #1}}
{\vspace{\the\theorempostskipamount}}
\def\example{\@ifnextchar[{\@exampleopt}{\@exampleplain}}
\def\@exampleplain{\paragraph{Example}}
\def\@exampleopt[#1]{\paragraph{Example \normalfont #1}}

\newenvironment{examples}[1][]
{\paragraph{Examples} {\normalfont #1}}
{\vspace{\the\theorempostskipamount}}
\def\examples{\@ifnextchar[{\@examplesopt}{\@examplesplain}}
\def\@examplesplain{\paragraph{Examples} \it}
\def\@examplesopt[#1]{\paragraph{Examples \normalfont #1}  \it}

\newenvironment{exercise}[1][]
{\paragraph{Exercise} {\normalfont #1}}
{\vspace{\the\theorempostskipamount}}
\def\exercise{\@ifnextchar[{\@exerciseopt}{\@exerciseplain}}
\def\@exerciseplain{\paragraph{Exercise} \it}
\def\@exerciseopt[#1]{\paragraph{Exercise \normalfont #1}  \it}

\newenvironment{notation}[1][]
{\paragraph{Notation} {\normalfont #1}}
{\vspace{\the\theorempostskipamount}}
\def\notation{\@ifnextchar[{\@notationopt}{\@notationplain}}
\def\@notationplain{\paragraph{Notation} \it}
\def\@notationopt[#1]{\paragraph{Notation \normalfont #1}  \it}

\newenvironment{convention}[1][]
{\paragraph{Convention} {\normalfont #1}}
{\vspace{\the\theorempostskipamount}}
\def\convention{\@ifnextchar[{\@conventionopt}{\@conventionplain}}
\def\@conventionplain{\paragraph{Convention} \it}
\def\@conventionopt[#1]{\paragraph{Convention \normalfont #1}  \it}

\newenvironment{warning}[1][]
{\paragraph{Warning} {\normalfont #1}}
{\vspace{\the\theorempostskipamount}}
\def\warning{\@ifnextchar[{\@warningopt}{\@warningplain}}
\def\@warningplain{\paragraph{Warning} \it}
\def\@warningopt[#1]{\paragraph{Warning \normalfont #1}  \it}

\newenvironment{theoreme}[1][]
{\paragraph{Théorème} {\normalfont #1} \it}
{\vspace{\the\theorempostskipamount}}
\def\theoreme{\@ifnextchar[{\@theoremeopt}{\@theoremeplain}}
\def\@theoremeplain{\paragraph{Théorème} \it}
\def\@theoremeopt[#1]{\paragraph{Théorème {\normalfont #1}}  \it}

\newenvironment{lemme}[1][]
{\paragraph{Lemme} {\normalfont #1} \it}
{\vspace{\the\theorempostskipamount}}
\def\lemme{\@ifnextchar[{\@lemmeopt}{\@lemmeplain}}
\def\@lemmeplain{\paragraph{Lemme} \it}
\def\@lemmeopt[#1]{\paragraph{Lemme {\normalfont #1}}  \it}

\newenvironment{de'finition}[1][]
{\paragraph{Définition} {\normalfont #1} \it}
{\vspace{\the\theorempostskipamount}}

\newenvironment{exemple}[1][]
{\paragraph{Exemple} {\normalfont #1}}
{\vspace{\the\theorempostskipamount}}
\def\exemple{\@ifnextchar[{\@exempleopt}{\@exempleplain}}
\def\@exempleplain{\paragraph{Exemple}}
\def\@exempleopt[#1]{\paragraph{Exemple {\normalfont #1}}}

\newcommand{\thmendspace}{\vspace{\the\theorempostskipamount}}
\newlength{\negvcorr}
\newlength{\mnegvcorr}
\makeatother

\newenvironment{proof}[1][Proof]{\bigskip\noindent \textit{#1.~}}
{\hfill $\Box$}

\newenvironment{mycpctitem}[1][---]
{\begin{list}{#1}
{\settowidth{\labelwidth}{#1}
\settowidth{\labelsep}{~}
\setlength{\leftmargin}{\labelwidth+\labelsep}
\setlength{\partopsep}{0cm}
\setlength{\topsep}{0cm}
\setlength{\parsep}{0cm}
\setlength{\itemsep}{0mm}}
}
{\end{list}}

\makeatletter

\def\@removefromreset#1#2{\let\@tempb\@elt
   \def\@tempa#1{@&#1}\expandafter\let\csname @*#1*\endcsname\@tempa
   \def\@elt##1{\expandafter\ifx\csname @*##1*\endcsname\@tempa\else
         \noexpand\@elt{##1}\fi}%
   \expandafter\edef\csname cl@#2\endcsname{\csname cl@#2\endcsname}%
   \let\@elt\@tempb
   \expandafter\let\csname @*#1*\endcsname\@undefined}

\@addtoreset{paragraph}{section}
\@addtoreset{paragraph}{part}
\@removefromreset{paragraph}{subsubsection}
\@removefromreset{paragraph}{subsection}

\let\c@equation\c@subparagraph

\makeatother

\renewcommand{\theparagraph}{(\arabic{section}.\arabic{paragraph})}
\renewcommand{\thesubparagraph}
{(\arabic{section}.\arabic{paragraph}.\arabic{subparagraph})}

\titleformat{\part}[display]{\normalfont\Large\bfseries}%
{\partname}{0cm}{}

\titleformat{\section}[hang]{\normalfont\Large\bfseries}{}{0cm}%
{\thesection \  --\ }

\titleformat{\subsection}[hang]{\normalfont\large\bfseries}{}{0cm}%
{\thesubsection \  --\ }

\newcommand{\spcifnec}[1]
{\ifx#1\empty
\else ~#1.
\fi}

\titleformat{\paragraph}[runin]{\normalfont\bfseries}
{\theparagraph}{0cm}{\spcifnec}
\titlespacing{\paragraph}{0cm}
{2.75ex plus 1ex minus .2ex}
{.5em}

\titleformat{\subparagraph}[runin]{\it}
{\thesubparagraph}{0cm}{\spcifnec}
\titlespacing{\subparagraph}{0cm}
{0mm}
{.5em}

\makeatletter
\let\coresubpar\subparagraph
\def\subparagraph{\@ifnextchar[{\@varsubpar}{\coresubpar}}
\def\@varsubpar[]#1{\coresubpar{}{\it \ \kern -.45em #1}}
\let\intersubpar\subparagraph
\def\subparagraph{\@ifnextchar*{\@starredsubpar}{\intersubpar}}
\def\@starredsubpar*{\@ifnextchar[{\@varstarredsubpar}{\@plainstarredsubpar}}
\def\@varstarredsubpar[]#1{\par\noindent {\it #1}}
\def\@plainstarredsubpar#1{\par\noindent {\it #1.}}

\let\corepar\paragraph
\def\paragraph{\@ifnextchar[{\@varpar}{\corepar}}
\def\@varpar[]#1{\corepar{}{\bf \ \kern -.45em #1}}
\let\interpar\paragraph
\def\paragraph{\@ifnextchar*{\@starredpar}{\interpar}}
\def\@starredpar*{\@ifnextchar[{\@varstarredpar}{\@plainstarredpar}}
\def\@varstarredpar[]#1{\bigskip\par\noindent {\bf #1}}
\def\@plainstarredpar#1{\bigskip\par\noindent {\bf #1.}}

\makeatother

{\titleformat{\section}[hang]{\normalfont\large\bfseries}{}{0cm}{}
\titleformat{\subsection}[hang]{\normalfont\bfseries}{}{0cm}{}
\renewcommand{\theparagraph}{(\Alph{paragraph})}

\maketitle
}{}

\newenvironment{closing}%
{\titleformat{\section}[hang]{\normalfont\large\bfseries}{}{0cm}{}
\setlength{\itemsep}{0mm}
\small
}
{}

\makeatletter
\renewcommand\@maketitle{%
  \newpage
  \begin{center}%
  \let \footnote \thanks
    {\Large \bf \@title \par}%
    \vskip 1em%
    {\large
      \begin{tabular}[t]{c}%
        \@author
      \end{tabular}\par}%
  \end{center}%
  \par
  \vskip 1.5em}
\makeatother

\renewenvironment{abstract}
{\small \quotation
\noindent {\bf Abstract.}}{\endquotation \vskip 1cm}

\newcommand{\N}{\mathbf{N}}
\renewcommand{\P}{\mathbf{P}}
\newcommand{\Gr}{\mathbf{Gr}}
\newcommand{\PH}{\mathbf{P}\kern -.05em \mathrm{H}}
\newcommand{\Aut}{\operatorname{Aut}}

\newcommand{\F}{\mathbf{F}}

\newcommand{\T}{\mathbf{T}}
\newcommand{\G}{\mathbf{G}}
\renewcommand{\O}{\mathcal{O}}
\newcommand{\I}{\mathcal{I}}

\renewcommand{\L}{\mathcal{L}}
\newcommand{\M}{\mathcal{M}}
\newcommand{\K}{\mathcal{K}}

\newcommand{\codim}{\mathrm{codim}}
\newcommand{\im}{\mathrm{im}}
\newcommand{\Cliff}{\mathrm{Cliff}}

\newcommand{\Pic}{\mathrm{Pic}}

\newcommand{\PGL}{\mathrm{PGL}}
\newcommand{\Aff}{\mathrm{Aff}} 

\newcommand{\cHom}{{\cal H}\kern -.08em om} 
\newcommand{\cExt}{{\cal E}\kern -.1em xt} 

\newcommand{\vect}[1]{\langle #1 \rangle} 
\newcommand{\lineq}{\sim} 
\renewcommand{\mod}{\ \textrm{mod}\ } 

\DeclareMathOperator{\cork}{cork}

\DeclareMathOperator{\expdim}{expdim}

\newcommand{\SL}{\mathrm{SL}}

\renewcommand{\L}{\mathcal{L}}

\let\congru\equiv

\newcommand{\dlbrack}{[ \kern -.4ex [}
\newcommand{\drbrack}{] \kern -.4ex ]}

\newcommand{\trsp}[1]{\vphantom{#1}^{\mathsf T\!} #1}
\newcommand{\restr}[2]{\left. #1 \right| _{#2}}

\makeatletter
\def\@orthpar[#1]{(#1)^\perp}
\def\@orthst#1{#1^\perp}
\def\orth{\@ifnextchar[{\@orthpar}{\@orthst}}
\makeatletter

\makeatletter
\def\@dualpar[#1]{(#1)^\vee}
\def\@dualst#1{#1^\vee}
\def\dual{\@ifnextchar[{\@dualpar}{\@dualst}}
\makeatletter

\newcommand{\bx}{\mathbf x}
\newcommand{\by}{\mathbf y}

\def\bT{\mathbf{T}}

\newcommand{\fl}{\mathfrak{l}}

\renewcommand{\epsilon}{\varepsilon}
\renewcommand{\geq}{\geqslant}
\renewcommand{\leq}{\leqslant}

\def\subset{\subseteq}

\newcommand{\ie}{i.e.,\ } 
\newcommand{\eg}{e.g.,\ }
\newcommand{\noeud}{n{\oe}ud}
\newcommand{\noeuds}{n{\oe}uds}
\makeatletter
\def\noeud{\@ifnextchar.{n{\oe}ud}{\@ifnextchar,{n{\oe}ud}{n{\oe}ud\ }}}
\def\noeuds{\@ifnextchar.{n{\oe}uds}{\@ifnextchar,{n{\oe}uds}{n{\oe}uds\ }}}
\makeatother
\def\?{?\kern -.08em ?}
\def\wtf{?\kern -.08em !}

\newcommand{\Kprim}{\mathcal{K}^{\mathrm{prim}}}
\newcommand{\Kcan}{\mathcal{K}^{\mathrm{can}}}
\newcommand{\KC}{\mathcal{KC}}

\newcommand{\Spin}{\mathrm{Spin}}
\newcommand{\LG}{\mathrm{LG}}

\newcommand{\Sp}{\mathrm{Sp}}
\def\Gr{\G}

\def\VW#1{\mathcal{V\kern -.1em W}_{#1}}

\let\ZZ\Z

\def\IC{\mathcal{IC}}

\def\sKC{K\kern -.15em C}
\let\CC\C

\let\SS\S
\def\S{\mathcal{S}}

\def\T{\mathcal{T}}

\title{{$K3$} curves with index {$k>1$}}
\author{Ciro Ciliberto and Thomas Dedieu}

\begin{document}


\renewcommand{\O}{\mathcal{O}}

\setdefaultenum{(i)}{}{}{}
\setdefaultitem{---}{}{}{}

\maketitle

\centerline{\emph{Dedicated to Fabrizio Catanese on the occasion of his 70th birthday,}} \centerline{\emph{with esteem, admiration and friendship.}}

\bigskip
\begin{abstract}
Let $\KC_g ^k$ be the moduli stack of pairs $(S,C)$ with $S$ a $K3$
surface and $C\subset S$ a genus $g$ curve with divisibility $k$ in
$\Pic(S)$. In this article we study the forgetful map $c_g^k:(S,C)
\mapsto C$ from $\KC_g ^k$ to $\M_g$ for $k>1$.
First we compute by geometric means the dimension of its general
fibre. This turns out to be interesting only when $S$ is a complete
intersection or a section of a Mukai variety. In the former case we
find the existence of interesting Fano varieties extending $C$ in its
canonical embedding. In the latter case this is related to delicate
modular properties of the Mukai varieties.
Next we investigate whether $c_g^k$ dominates the locus in $\M_g$ of
$k$-spin curves with the appropriate number of independent sections.
We are able to do this only when $S$ is a complete intersection, and
obtain in these cases some classification results for spin curves.
\end{abstract}

\section{Introduction}

We are interested in the forgetful maps
$c_g ^k : \KC_g ^k \to \M_g$
with $k>1$ and $g=1+(g_1-1)k^2$ for some integer $g_1>1$,
in the following notation (valid also in case $k=1$):
\smallskip
\begin{mycpctitem}
\item  $\K_g^k$ is the moduli stack of polarised $K3$ surfaces of
  genus $g$ and index $k$, 
\ie pairs $(S,L)$ such that $S$ is a smooth $K3$ surface and $L$ is
an ample, globally generated
line bundle on $S$ with $L^ 2=2g-2$, such that $L=kL_1$ with $L_1$ a
primitive line bundle on $S$;
note that $(S,L_1)$ is a membre of $\K_{g_1} ^1$, which we usually
denote by $\Kprim _{g_1}$;
\item $\KC_g^k$ is
the moduli stack of pairs $(S,C)$
with $C$ a smooth curve on $S$ and $(S,\O_S(C))\in \K_g^k$;
\item $\M_g$ is the moduli space of curves of genus $g$.
\end{mycpctitem}
Specifically our goal is to describe the general fibre and the image
of the maps $c_g^k$.

In the primitive case $k=1$ the situation is rather well understood
and should be well-known by now, so we don't dwell on this here
(the interested reader may consult, \eg \cite[\SS 2]{cds} to find the
relevant references), and concentrate on the case $k>1$.

If $C$ is a smooth curve of genus $g\geq 11$ and Clifford index
$\Cliff(C)>2$ lying on a $K3$ surface, it follows from work of
Wahl \cite{wahl97},
Arbarello--Bruno--Sernesi \cite{abs1},
and eventually the authors and Sernesi \cite{cds}
that the fibre of the map
\[
  \textstyle
  c_g: [S,C] \in \coprod_k \KC_g^k \longmapsto [C] \in \M_g
\]
(which is the aggregation of the maps $c_g^k$ for all $k$ such
that there exists $g_1 >1$ with $g=1+(g_1-1)k^2$)
is essentially a linear space, of dimension
$\nu =\cork(\Phi_C)-1$ (where $\Phi_C$ is the Gauss--Wahl map of $C$,
see \cite{cds} for background and references),
and there exists
an arithmetically Gorenstein normal variety
$X \subset \P^{g+\nu}$ of dimension $\nu+2$ with
$\omega_X=\mathcal O_X(-\nu)$,
having as a linear section the curve $C$ in its canonical embedding,
as well as all $K3$ surfaces $S$ containing $C$, embedded by the linear
system $|C|_S$; we call such an $X$ the \emph{universal extension} of
$C$.
(Recall that an \emph{$r$-extension} (or simply extension, if $r=1$) of
a projectively embedded variety $V \subset \P^N$ is a variety
$X \subset \P^{N+r}$ having $V$ as a linear section; an extension is
said to be trivial if it is a cone; again, see
\cite{cds} for background and references).

We call \emph{$K3$ curve of index $k$} (and genus $g$)
a curve $C$ which is in the image of the map $c_g^k$.
For general such curves with $k>1$ the corank of the Gauss--Wahl map
has been computed by cohomological methods in \cite{CLM98},
with some exceptions for $g_1=2$ documented in a corrigendum.
In this article we compute by geometric considerations the dimension
of the general fibre of the maps $c_g^k$, $k>1$, thus concluding the
task begun in \cite{cd-double} with the case $g_1=2$.

It turns out to be interesting only for low values of $g_1$, for which
the $K3$ surfaces are either complete intersections or sections of
more general homogeneous varieties found by Mukai and thence called
\emph{Mukai varieties} (see Section~\ref{S:Mukai}).
Indeed, it follows from Prokhorov's bound on the genus of Gorenstein
Fano threefolds \cite{prokhorov} that the maps $c_g^k$ are generically
injective if $g>37$, see \cite[Cor.~2.10]{cds}.

\begin{proposition}
\label{pr:condCDS}
Let $k>1$ and $g_1>1$ be integers, and set 
$g=1+(g_1-1) k^2$.
Let $[C]$ be very general in the image of $c _g ^k$.
\subparagraph{}
\label{pr:table-i}
We have $\Cliff(C) >2$ and $g(C)\geq 11$ unless we are in one of the
following cases:
\begin{mycpctitem}
\item $g_1=2$ and $k\leq 3$;
\item $g_1=3$ and $k=2$.
\end{mycpctitem}
\subparagraph{}
\label{pr:table-ii}
We have $g(C)>37$ unless we are in one of the
following cases:
\begin{center}
\begin{tabular}{c|lllllllll}
$g_1$    & $2$ & $3$ & $4$ & $5$ & $6$ & $7$ & $8$ & $9$ & $10$
\\ \hline
\raisebox{0mm}[3.5mm]{$k \leq$} & $6$ & $4$ & $3$ & $3$ & $2$
 & $2$ & $2$ & $2$ & $2$
\end{tabular}
\end{center}
\end{proposition}

\bigskip
Part~\ref{pr:table-i} follows from Lemma~\ref{l:cliff} and it tells us that
except for the exceptional cases
$(g_1,k) \in \{ (2,2), (2,3), (3,2) \}$
the general theory of \cite{wahl97,abs1,cds} applies.
In the exceptional cases we have been able to analyze the situation by
hand in \cite{cd-double}; we obtain a little more information here.
Part~\ref{pr:table-ii} follows from a direct computation, and tells us
the values of $(g_1,k)$ for which it is possible that $c_g^k$ has
positive dimensional general fibre.
We now give a brief summary of the results we obtain in these cases.

\paragraph{Complete intersections
{\normalfont (Section~\ref{S:CI})}}
For $g_1 \in \{3,4,5\}$ the general membre $C$ of the image of $c_g^k$, with 
$g=1+(g_1-1)k^2$, has a model as a complete intersection in
$\P^{g_1}$. This allows for a direct description of the fibre
$(g_g^k)^{-1}(C)$; in particular we can compute its dimension, which
gives $\cork(\Phi_C)$, thus recovering the values obtained in
\cite{CLM98}.

In all cases, we provide a description of the universal extension $X$
as a complete intersection in a weighted projective space, embedded by
a suitable divisor in $\Pic(X)$ of the anticanonical class.
In fact this also works in the exceptional cases of \ref{pr:table-i},
in which cases we obtain interesting examples showing the sharpness of
Lvovski's main theorem in \cite{lvovskiII}.
This construction has been inspired by the corresponding one in case
$g_1=2$, which has been communicated to us by Totaro
(see \cite[(4.8)]{cd-double}).

\paragraph{Mukai models
{\normalfont (Section~\ref{S:Mukai})}}
For $g_1 \in \{6,7,8,9,10\}$ the general membre $C$ of the image of
$c_g^k$, with $g=1+(g_1-1)k^2$, has a model as a complete intersection in
a Mukai homogeneous variety $M_{g_1}$.
For $g_1 \neq 6$ in this range, we show by
geometric considerations that the maps $c_g^k$ are generically
injective for all $k>1$. For $g_1=6$ and $k=2$ the dimension of the
general fibre is $1$, see \ref{p:g6.dimfibre}.
On our way we obtain results about the related question whether all
models of $C$ are 
conjugate under the action of $G=\Aut(M_{g_1})$,
see Conjecture~\ref{conj:Mukai} and summary of results in
\ref{p:mukai.results}. 

We also show that in cases
(a) $k=1$ and $g_1>6$, and (b) $k=2$ and $g_1=6$,
the Mukai variety is the universal extension of the general membre of
$\im(c_g^k)$; this was certainly a natural thing to guess, but as far
as we know this had not been proved yet.

A more detailed but still synthetic description of our results about
Mukai varieties and their sections is given in
subsection~\ref{s:mukai.intro}.

\bigskip
Our study (in Section~\ref{S:spin})
of the image of the maps $c_g^k$ on the other hand
is based on the following observation.

\paragraph{$K3$ curves of index $k$ and $k$-spin curves}
For $g=1+k^2(g_1-1)$, the map $c_g^k$ factorizes through 
\[
(S,C) \in \KC_g ^k
\longmapsto
\textstyle (C,\restr {\O_S(\frac 1 k C)} C)
\in \S _g ^{\frac 1 k, g_1},
\]
where
\begin{mycpctitem}
\item $\S _g ^{\frac 1 k, g_1}$ is the moduli stack of
\emph{$k$-spin curves}
of genus $g$ having at least $g_1+1$ independent sections,
\ie pairs $(C,\theta)$ with $C$ a smooth genus $g$ curve and $\theta$ a line
bundle on $C$ such that $k\theta=K_C$ and $h^0(\theta)\geq g_1+1$.
\end{mycpctitem}
A line bundle $L$ such that $2L=K_C$ is called a
theta-characteristic (hence the notation).
Moreover we denote by
\begin{mycpctitem}
\item 
\raisebox{0mm}[5mm]{$\T _g ^{\frac 1 k, g_1}$} the image of 
$\S _g ^{\frac 1 k,
  g_1}$ in $\M_g$ by the (finite) forgetful map $(C,L)\mapsto C$,
\end{mycpctitem}
so that
\begin{equation}
\label{eq:r-spin-img}
\im(c _g ^k) \subset \T _g ^{\frac 1 k, g_1}.
\end{equation}
The question underlying our study in Section~\ref{S:spin} is whether
equality holds in \eqref{eq:r-spin-img}.

For $k=2$  there is an
expected codimension for
$\T _g ^{\frac 1 2, g_1}$ in $\M_g$, viz.\ $\binom {g_1+1} 2$;
more precisely, the following holds.

\begin{theorem}[(\cite{mumford} and {\cite[Cor.~1.11]{harris-theta}})]
\label{t:harris}
Let $(C,\theta) \in \S _g ^{\frac 1 2,  g_1}$ be such that
$h^0(\theta)=g_1+1$. Then for any deformation $(C_t,\theta_t)$ of $(C,\theta)$ one has
$h^0(\theta_t) \congru h^0(\theta) \mod 2$, and the codimension $c$ of 
$\S _g ^{\frac 1 2,  g_1}$ in $\S _g ^{\frac 1 2}$
at $(C,\theta)$ is $c \leq \binom {g_1+1}2$.
\end{theorem}

\medskip
However several examples are known of superabundant components of 
$\S _g ^{\frac 1 2,  g_1}$, see \eg \cite{farkas-theta}.
Here we observe that $K3$ curves of index
$2$ provide infinitely many such examples.

\paragraph{Observation}
We have the following remarkable formula:
\begin{align}
\notag
\dim (\KC_g ^2)-\expdim(\T _g ^{\frac 1 2, g_1})
&= \textstyle
(19+g)-(3g-3)+\binom {g_1+1}2 \\
\label{expdim-spin} &= \textstyle
\frac 1 2 (g_1-7)(g_1-8).
\end{align}
Since for $g_1 \geq 7$ the map $c_g^2$ is generically injective
(Corollary~\ref{c:cork}), this
formula implies that the  image of $c_g^2$ lies in a superabundant
component of $\S _g ^{\frac 1 2,  g_1}$ as soon as $g_1\geq 9$.

Moreover this difference equals the dimension of the general fibre of
$c_g^2$ for $3 \leq g_1 \leq 6$ (see Section~\ref{S:CI} and
paragraph~\ref{p:g6.dimfibre}). For $3 \leq g_1 \leq 5$ the closure of
$\im(c_g^2)$ is an irreducible component of $\T _g ^{\frac 1 2, g_1}$
(Proposition~\ref{prop:compint}) which thus has the expected dimension.

This makes the question whether the
closure of $\im(c_g^2)$ is an irreducible component of
$\T _g ^{\frac 1 2, g_1}$ for $g_1=6,7,8$ rather intriguing.

In the complete intersection cases, we can prove equality in
\eqref{eq:r-spin-img} (see Proposition \ref {prop:compint} below), and then it is possible to check that 
$\T _g ^{\frac 1 2, g_1}$ has a component of the expected dimension.

\paragraph{Classification of $K3$-like spin curves
{\normalfont (Section~\ref{S:spin})}}
We make a thorough analysis of the components of $\S _g ^{\frac 1 2,
g_1}$ in the cases $g_1=3$ and $g_1=4$ (Theorems \ref {thm:g3} and
\ref {prop:comp}), and give some partial information for $g_1=5$ (see
Proposition \ref {prop:g1=5}). In the absence of an estimate for the
codimension of $\T _g ^{\frac 1 k, g_1}$ in $\M_g$ for $k\geq 3$, the
cases $k\geq 3$ are more complicated, and we give only some partial
answers for $k=3$ and $3\leq g_1\leq 4$ (see Section \ref {sec:k3}).

This turns out to be unexpectedly involved, and is the occasion of
discovering several interesting families of curves on rational normal
scrolls in $\P^{g_1}$ for $g_1=3,4,5$.

\bigskip
\noindent
\textbf{Thanks.}
ThD thanks Laurent Manivel for his answers to the many questions he
was asked about homogeneous varieties,
and more generally for his interest and positive influence on this
project.
Additional thanks to Marian Aprodu for his help with the proof of
Lemma~\ref{l:mukai-cliff}.

\section{Preliminaries}
\label{S:prelim}

\paragraph{Notation}
When the context is clear, we use the shorthand
$h_n(k)$ for $h^0(\P^n,\O_{\P^n}(k))$.

\begin{lemma}
\label{l:cliff}
Let $(S,L_1)\in \Kprim_{g_1}$ be very general, so that
$\Pic(S)={\bf Z} \langle L_1\rangle$. 
For smooth $C \in |kL_1|$, $k>1$, we have
$\Cliff(C) = (2g_1-2)(k-1)-2$.
\end{lemma}

\begin{proof}
Let $C$ be a smooth membre of $|kL_1|$; it has genus
$g = 1+k^2(g_1-1)$.
By \cite{green-lazarsfeld}, the Clifford index of $C$ is either
$\lfloor \frac {g-1}2 \rfloor$ or it 
is computed by the restriction of a line bundle on $S$. 
Since $\Pic(S)={\bf Z} \langle L_1\rangle$, we have to consider those $lL_1 \in \Pic(S)$
such that $h^0(\restr {lL_1} C) \geq 2$ and $h^0(K_C-\restr {lL_1} C)\geq 2$,
which amounts to the condition that $1 \leq l <k$.
One computes
\begin{align*}
\Cliff(\restr {lL_1} C) 
= \deg(\restr {lL_1} C)-2(h^0(\restr {lL_1} C) -1)
= (2g_1-2)l(k-l)-2,
\end{align*}
the minimal value of which for integral $l \in [1,k-1]$ is 
$(2g_1-2)(k-1)-2$,
indeed less than $\lfloor \frac {g-1}2 \rfloor$.
\end{proof}

\bigskip
In this article we use freely the notion of ribbon over a curve and
its relation with the Gauss--Wahl map;
we refer to \cite[\SS 4]{cds} and \cite[\SS 1.2]{cd-double} for the
necessary background.

\begin{proposition}
\label{pr:unicity-integral}
Let $C$ be a curve with $\Cliff(C)>2$. For any 
$v \in \ker (\trsp \Phi_C)$ seen as a ribbon over $C \subset \P ^{g-1}$, 
there exists at most one surface extension $S \subset \P^g$ up to
projectivities
(\ie the integral of $v$, if it exists, is unique).
\end{proposition}

\begin{proof}
It is given in \cite[Remark~4.8]{cds}.
There one finds the additional assumption that $g\geq 11$, but this is
useless as far as unicity is concerned (it is needed only to ensure
the existence of an integral). 

What is really needed is that $C \subset \P^{g-1}$ is defined by
quadratic equations with linear syzygies, which is ensured by the
condition on the Clifford index. Under this assumption, one can indeed
work out Wahl's deformation construction, and then one sees 
\cite[\SS 4.9]{cds} that any two
extensions differ by an element of $H^0(N_{C/\P^{g-1}}(-2))$.
The conclusion stems from the fact that the latter cohomology group is
$0$ if $C \subset \P^{g-1}$ is defined by quadratic equations with
linear syzygies (see \cite[Lemma 3.6]{cds}).
\end{proof}

\begin{proposition}
\subparagraph{}
\label{pr:finite-auto-pol}
Let $(S,L)$ be a polarized $K3$ surface.
Then the automorpshisms group of $(S,L)$ is finite.
\subparagraph{}
\label{pr:no-auto-vgK3}
Let $S$ be a projective $K3$ surface with 
$\Pic(S) = \ZZ\langle L_1\rangle$. Then
$\Aut(S)$ is trivial if $(L_1)^2>2$, and
isomorphic to $\ZZ/2$ if $(L_1)^2=2$.
\end{proposition}

\bigskip\noindent
See \cite[Cor.~15.2.12]{huybrechts} for \ref{pr:no-auto-vgK3},
and \cite[Chap.~5 Prop.~3.3]{huybrechts} for
\ref{pr:finite-auto-pol}.

\section{Complete intersections}
\label{S:CI}

These are the cases $g_1=3,4,5$.

\subsection{Useful results}

\begin{theorem}
\label{p:CI}
Let $Z,Z'$ be two complete intersections in some projective space
$\P^n$.
\subparagraph{}
\label{p:moduli-ic}
The two polarized schemes $(Z,\O(1))$ and $(Z',\O(1))$ are
isomorphic if and only if they are conjugate under the action of
$\PGL(n+1)$ on $\P^n$.
\subparagraph{}
\label{p:stab-ic}
If $Z$ is non-degenerate and smooth (resp.\ general), 
then $\Aut(Z,\O(1))$ is finite
(resp.\ trivial)
unless $Z$ is a quadric hypersurface
(resp. unless it is either a quadric hypersurface or the complete
intersection of two quadrics).
\end{theorem}

\bigskip
To be noted that $\Aut(Z,\O(1))$ is the stabilizer of $Z$ in
$\PGL(n+1)$. 
For \ref{p:moduli-ic}, see \cite[Prop.~2.1]{benoist-these}.
Part \ref{p:stab-ic} in the smooth case is
\cite[Thm.~1.6]{benoist-separation}, and in the general case it is
\cite[Thm.1.3]{cpz-MM}.

\subsection{Fibres and universal extensions}

\paragraph{$g_1=3$: dimension of the general fibre}
\label{par:fibre}
Let $k \in \N\setminus \{0\}$, and consider the forgetful map
$c_{g} ^k: \KC_{1+2k^2} ^k \to \M _{1+2k^2}$,
$g=1+2k^2$.
Its image contains $\IC(4,k)$ as a dense subset (the
locus of smooth complete intersection curves of type $(4,k)$ in
$\P^3$).

Consider a general curve $[C] \in \IC(4,k)$. Up to
the action of $\PGL(4)$, there is a unique\footnote
{unique indeed if we consider $C$ as a curve with a $k$-spin
structure, but in fact a finite number if we only take it as a naked
curve} 
incarnation of $C$ as a
complete intersection in $\P^3$, which we call $C$ as well.
The linear system of quartics containing $C$ has dimension
\[
h^0(\I_C(4))-1 = h_3(4-k)
\]
and dominates the fibre of $c_g^k$ over $C$.
In fact, the quotient of this linear system by the stabilizer of $C$
in $\PGL(4)$ is birational to this fibre.
When $k>1$ this stabilizer is trivial, and when $k=1$ it is the affine
group $\Aff(3)$ which has dimension $4$
(and indeed $h_3(3)-4=16$, the dimension of the general fibre of
$c^1_3$ which is dominant).
We thus find the values indicated in the table below for the dimension
of the general fibre of $c_g^k$.

\begin{center}
\begin{tabular}{r|llllll}
$g_1=3$ \\
$k$ & 2 & 3 & 4 & $\geq 5$ \\
\hline
$g$    & $9$ & $19$ & $33$ & $\geq 51$
\\ 
$\dim(\text{fibre})$ 
& 10 & 4 & 1 & 0
\end{tabular}
\end{center}

When $k\geq 3$ we have $g\geq 11$ and $\Cliff>2$ so we may apply the
theory in \cite {cds}, and this tells us the value of $\cork (\Phi_C)$ for the
general $[C] \in \IC(4,k)$, as well as the number of times $C$ may be
extended in its canonical embedding.
When $k \leq 2$ however there is little we can do but observe the
coincidence in dimension with
$\sum _{l\geq 1} h^0(N(-l))$, for $C$ does not enjoy property $N_2$
(even in its canonical embedding).

\paragraph{$g_1=3$: universal extensions}
\label{p:g3.univ}
We now describe the universal extensions of canonical curves in the
image of $c_g^k$ for $2\leq k\leq 4$
(when $k>4$ the universal extension is the unique quartic surface $S
\subset \P^3$ containing $C$, re-immersed in $\P^g$ by the linear
system $|\O_S(k)|$).
They are similar to the extensions found
by Totaro (private communication) in the case $g_1=2$
(see \cite[(4.8)]{cd-double} and \ref{p:totaro}
below).
The case $k=2$ is somewhat exceptional, 
as $h^0(N_{C/\P^9}(-2))=1$ hence there is no universal extension guaranteed
by \cite{cds}
(see also Remark~\ref{r:sharp-Lv.g3}); still it works out as the others.
{\setlength{\parindent}{0mm}
\begin{asparaitem}
\item $k=4$: the universal extension is $\P^3$, in its anticanonical
embedding $v_4(\P^3) \subset \P^{34}$; it may also be presented as a
quartic hypersurface
$X_{4,4} \subset \P(1^4,4)$ in its anticanonical embedding, in $\P^{34}$
by $|\O(4)|$
($|\O(4)|$ embeds $\P(1^4,4)$ in $\P^{35}$ as the cone over the
$4$-Veronese $v_4(\P^3)\subset \P^{34}$).
\item $k=3$: the universal extension has dimension $6$, embedded in
$\P^{23}$ with index $4$; it is again a quartic $X_{4,3}$, this time in
$\P(1^4,3^4)$; the anticanonical of $X_{4,3}$ is $\O(12)$, correspondingly
the embedding with index $4$ is given by (weighted) cubics
(and indeed $h^0(\O_{X_{4,3}}(3))=24$).
Specifically the equation is as follows. Assume $C \subset \P^4$ is
the complete intersection of the quartic $(f=0)$ and the cubic
$(g=0)$. Then $X_{4,3}$ is defined by
\[
  f(\bx)+\bx\cdot \by
  = f(x_0,\ldots,x_3)+x_0y_0+\cdots+x_3y_3
  =0,
\]
where the $x_i$'s and $y_i$'s are the homogeneous coordinates of
weight $1$ and $3$ respectively.
Indeed the general quartic surface in $\P^3$ containing $C$ has an equation
of the form $f+(a_0x_0+\cdots+a_3x_3)g=0$, hence can be obtained by
cutting $X_{4,3}$ by the $4$ cubic equations
\[
  y_0-a_0g(\bx)
  = \cdots
  = y_3-a_3g(\bx) =0.
\]
One gets $C$ by cutting by the fifth cubic $g(\bx)=0$.
\item $k=2$: we can package together all the quartic surfaces containing a
 complete intersection curve of type $(2,4)$ as a (weighted) quartic
hypersurface $X_{4,2}$ in $\P(1^4,2^{10})$. This has anticanonical
$\O(20)$, hence may be embedded with index $10$ by (weighted)
quadrics;
one has $h^0(\O_{X_{4,2}}(2))=20$, so one ends up with
$X_{4,2} \subset \P^{19}$ as required.
Take $C$ the complete intersection of the quartic $(f=0)$ and the
quadric $(g=0)$. Then the equation of $X_{4,2}$ is
\[
  f(\bx)+ \bigl(x_ix_j\bigr)_{0 \leq i\leq j\leq 3}\cdot \by
  = f(\bx)+x_0^2y_0+x_0x_1y_1+\cdots + x_3^2y_{10}
  =0.
\]
The general quartic homogeneous polynomial vanishing along $C$ is of the
form $f+qg$ for some quadratic form
$q(\bx)=a_0x_0^2+\cdots+a_{10}x_3^2$, and one gets the corresponding
surface in $\P^3$ by cutting by the $10$ degree $2$ equations
\[
  y_0-a_0g(\bx)=\cdots=y_{10}-a_{10}g(\bx)=0.
\]
\end{asparaitem}
}

\begin{remark}
\label{r:sharp-Lv.g3}
The case $k=2$ above provides interesting insight into
Lvovski's main theorem in \cite{lvovskiII}. He proves that
if $X \subset \P^n$ is a smooth, non-degenerate variety which is not a
quadric, then $X$ cannot be extended non-trivially more than
$\alpha(X)$ times if $\alpha(X)<n$, with
$\alpha(X)=h^0(N_{X/\P^n}(-1))-n-1$. 

Consider a smooth complete intersection curve $C\subset \P^3$ of type
$(2,4)$, in its canonical embedding in $\P^8$.
The construction above proves that $C$ is extendable $11$ times, to
the $12$-fold $X_{4,2} \subset \P^{19}$, whereas
\[
  \alpha(C) = \cork(\Phi_C)=10.
\]
Thus $C \subset \P^8$ is extendable strictly more 
than $\alpha(C)$ times;
indeed the assumption ``$\alpha(C)<8$'' of Lvovski's theorem
does not hold. Note moreover that $C \subset \P^8$ is not a complete
intersection. This shows (a) that the condition ``$\alpha(X)<n$'' is stronger
than ``$X \subset \P^n$ is not a complete intersection'', and
(b) that one may not replace the assumption
``$\alpha(X)<n$'' by ``$X \subset \P^n$ is not a complete intersection''
in Lvovski's theorem.

On the other hand it follows from Lvovski's theorem that the
two Fano's $X_{4,4} \subset \P^{34}$ and
$X_{4,3} \subset \P^{23}$ constructed in \ref{p:g3.univ}
are not extendable.
\end{remark}

\bigskip
The other complete intersection cases are similar, so we will give
less details.

\paragraph{$g_1=4$}
Let $k \in \N\setminus \{0\}$, and consider the forgetful map
$c_{g} ^k: \KC_{1+3k^2} ^k \to \M _{1+3k^2}$,
$g=1+3k^2$.
Its image contains $\IC(2,3,k)$ as a dense subset (the
locus of smooth complete intersection curves of type $(2,3,k)$ in
$\P^4$).

\begin{center}
\begin{tabular}{r|llllll}
$g_1=4$ \\
$k$ & 2 & 3 & $\geq 4$ \\
\hline
$g$    & $13$ & $28$ & $\geq 49$
\\ 
$\dim(\text{fibre})$ 
 & 6 & 1 & 0
\\
threefold & cubic & quadric & no
\end{tabular}
\end{center}

For $k=2$ we consider complete intersection curves $C$ of type $(2,2,3)$
in $\P^4$. Then $h^0(\I_C(2))=2$ and  $h^0(\I_C(3))=11$; there is a
$\P^1$ of quadrics containing $C$, and for each of these the $\P^{10}$
of cubics containing $C$ cuts out a $\P^5$ of $K3$'s complete
intersection of type $(2,3)$. Eventually the fibre $(c_{13}^2)^{-1}(C)$ has
dimension $6$.
The universal extension is a weighted complete intersection of type
$(2,3)$,
$X_{(2,3),2} \subset \P(1^5,2^6)$; it has anticanonical $\O(6\cdot 2)$, hence
it is embedded with index $6$ in $\P^{19}$ by the linear system of
weighted quadrics. 
The equations of  $X_{(2,3),2}$ are as follows:
let $f=g=g'=0$ be the equations of $C$,  with
$\deg(f)=3$ and $\deg(g)=\deg(g')=2$;
then $X_{(2,3),2}$ may be defined by the equations
\[
  \left\{
  \begin{aligned}
    g(\bx)+y_5 &=0 \\
    f(\bx)+ x_0\cdot y_0 + \cdots + x_4 \cdot y_4 &= 0.
  \end{aligned}
\right.
\]
with the $x_i$'s and $y_i$'s of weights $1$ and $2$ respectively.
A complete intersection $K3$ in $\P^4$ containing $C$ is the
intersection of a quadric of the form $g+a_6g'=0$ and a cubic
of the form $f+(a_0x_0+\cdots+a_4x_4)g'$, which is obtained by cutting
$X_{(2,3),2}$ by $6$ quadric equations $y_5-a_6g'=0$, etc.

For $k=3$,
a smooth complete intersection $C$ of type $(2,3,3)$ is contained in a
unique quadric, and there is a pencil of cubics containing $C$ and
independent of the quadric, hence a pencil of $K3$ surfaces containing
$C$. 
The universal extension is a $3$-fold, namely the
quadric in $\P^4$, in its anticanonical embedding $v_3(X_2) \subset
\P^{29}$.
This may also be presented in a uniform fashion with the other
extensions: a quadric in $\P^4$ is
also a complete intersection $X_{(2,3),3}$ of type $(2,3)$ in
$\P(1^5,3)$.

It follows from Lvovski's theorem that
$X_{(2,3),2} \subset \P^{19}$
and 
$X_{(2,3),3} \subset \P^{29}$
are not extendable.

\paragraph{$g_1=5$} In this case we consider complete intersection
curves of type $(2^3,k)$ in $\P^5$.

\begin{center}
\begin{tabular}{r|llllll}
$g_1=5$ \\
$k$ & 2 & $\geq 3$ \\
\hline
$g$    & $17$ & $\geq 37$
\\ 
$\dim(\text{fibre})$ 
 & 3 & 0
\end{tabular}
\end{center}

For $k=2$ the curve $C$ is a complete intersection of type $(2^4)$ in
$\P^5$. The $K3$ surfaces extending it correspond to the planes in the
$3$-space $\P (H^0(\I_C(2)))$, the family of which is a $\P^3$.
The universal extension is a Fano $5$-fold $X_{(2^3),2}$ complete
intersection of type $(2^3)$ in $\P(1^6,2^3)$,
with anticanonical $\O(3\cdot 2)$ and embedded by quadrics in
$\P^{20}$.
For $C$ with equations $g_0=g_1=g_2=g_3=0$, where $g_i$ are forms of
degree 2 for $0\leq i\leq 3$, we may take for
$X_{(2^3),2}$ the equations 
\[
  \left\{
    \begin{aligned}
      g_0(\bx)+y_0 &= 0 \\
      g_1(\bx)+y_1 &= 0 \\
      g_2(\bx)+y_2 &= 0,
    \end{aligned}
  \right.
\]
with the $x_i$'s and $y_i$'s of weights $1$ and $2$ respectively.
Lvovski's theorem implies that $X_{(2^3),2} \subset \P^{20}$
is not extendable.

For $k \geq 3$ there is a unique net of quadrics containing $C$, hence
$C$ is extendable only once.

\begin{remark}
\label{p:totaro}
The case of sextic double planes ($g_1=2$, which we studied in details
in \cite{cd-double}) is closely analogous to
that of quartics in $\P^3$, as indeed a
sextic double plane is a sextic hypersurface in $\P(1^3,3)$ so in both
cases we deal with hypersurfaces.
For $g_1=2$, the dimension of the general fibre of $c_g^k$ is given in
the table below (see \cite{cd-double}).
\begin{center}
\begin{tabular}{r|lllllllll}
$g_1=2$ \\
$k$ & 2 & 3 & 4 & 5 & 6 & $\geq 6$ \\
\hline
$g$ & 5 & 10 & 17 & 26 & 37 
\\ 
$\dim(\text{fibre})$
    & 15& 10 & 6 & 3 & 1 & 0
\end{tabular}
\end{center}
For $k \leq 6$ we can package together the $K3$ surfaces extending a
curve $C$ in the image of $c_g^k$ as a Fano $(2+\nu)$-fold of index
$\nu$ in $\P^{g+\nu}$, where $\nu = \dim ((c_g^k)^{-1}(C))$, which is
the image of a sextic hypersurface in
$X_{6,k} \subset \P(1^3,3,k^\nu)$ (with equation similar to that given in
\ref{p:g3.univ}) by the linear system of weighted $k$-ics.

For $k \leq 2$ this linear system is hyperelliptic, \ie it gives a
$2:1$ map.
For $k=3$, we have $\cork(\Phi_C)=10$ and $h^0(N_{C/\P^9}(-2))=1$, and
$X_{6,3}$ is an example similar to that in Remark~\ref{r:sharp-Lv.g3},
which shows the sharpness of Lvovski's theorem.
For $k \leq 4$ the Fano $(2+\nu)$-fold $X_{6,k} \subset \P^{g+\nu}$ is
not extendable, by Lvovski's theorem. For $k=6$ in fact $X_{6,6}$ is
merely $\P(1^3,6)$ itself in its anticanonical embedding.
\end{remark}

\section{Curves on Mukai varieties}
\label{S:Mukai}

In this section we consider the cases 
$g_1=6,7,8,9,10$.

\subsection{Introduction and a conjecture}
\label{s:mukai.intro}

\paragraph{}
\label{p:mukai.descr}
For $g_1=7,8,9,10$, there exists a variety
$M_{g_1} \subset \P(U_{g_1})$,
homogeneous for some simple algebraic group $G$ of which
$U_{g_1}$ is an irreducible representation, such that the
general prime $K3$ surface of genus $g_1$ is
a linear section of $M_{g_1}$
\cite[Cor.~0.3]{mukai1}.
These groups, representations and homogeneous varieties are recalled
in the table below.

\begin{center}
\begin{tabular}{l|llllll|l}
\(g\) & \(G_g\) & \(\dim(G)\) & \(U_g\) & \(\dim(U_g)\) & \(M_g\) & \(\dim(M_g)\) & \(k_g\)\\
\hline
\(7\) & \(\Spin_{10}\) & 45 & \(\Delta_+\) & 16 & \(S_{10}\) & \(10\) & 4\\
\(8\) & \(\SL_6\) & 35 & \(\wedge^2\CC^6\) & 15 & \(\G(2,6)\) & \(8\) & 3\\
\(9\) & \(\Sp_6\) & 21 & \(\wedge^{\langle 3\rangle}\CC^6\) & 14 & \(\LG(3,6)\) & \(6\) & 2\\
\(10\) & \(G_2\) & 14 & \(\mathfrak{g}_2\) & 14 & \(G_2/P_2\) & \(5\) & 2\\
\end{tabular}
\end{center}
\emph{Notation:} we set $n(g)=\dim(U_{g}) -1$.
We will sometimes drop the subscript $g$ in order to lighten the
notation. \medskip

We will not use directly all the information in this table, so the reader
unfamiliar with representation theory may safely read only those columns
giving dimensions.

Two smooth surface linear sections of
$M_{g_1}$ are projectively equivalent if and only if  one is image
from the other by an automorphism coming from $G$
\cite[Thm.~0.2]{mukai1}.
The proof of the latter statement may be adapted to prove the
same result for smooth linear sections of higher dimension, but it
breaks down in the case of curve linear sections of
$M_{g_1}$.

In the present article we use ribbons to transport the results proved
by Mukai for surface sections of $M_{g_1}$ to curve sections,
see \ref{conj:Mukai} and \ref{p:mukai.results} below.

A curve in the image of $c_{g} ^k$ is a complete intersection $C$ of
type $(1^{m-2},k)$ in $M_{g_1}$
($m=\dim(M_{g_1})$), hence for $k>1$ there is a unique surface linear
section of $M_{g_1}$ containing $C$. The description of the
fibre $(c_{g} ^k)^{-1}(C)$ depends on the truth of the
Conjecture~\ref{conj:Mukai} below. 
Before we state it, let us describe the analogous results for
$g_1=6$.

\paragraph{}
For $g_1=6$
there is the Mukai--Fano threefold $V \subset \P^6$
which is a linear section of $\Gr(2,5) \subset \P^9$ in its Plücker
embedding; it has index $2$ and degree $5$.
The generic prime $K3$ surface of genus $6$ is a section of $V$ by a
quadric in $\P^6$;
one has 
$\Aut(V) \cong \PGL(2)$, and two smooth quadric sections of $V$ are
projectively isomorphic if and only if one is image from the other by
an automorphism of $V$
\cite[Thm.~4.1]{mukai1}.

Smooth surface linear sections of $V$ all are Del Pezzo surfaces of
degree $5$, in particular they are all isomorphic.
Moreover, a general genus $6$ curve $C$ in its canonical embedding is
a quadric section of the Del Pezzo surface of degree $5$ in $\P^5$,
and therefore has a $6$-dimensional family of models as a complete
intersection of type $(1,2)$ on $V$; since $\Aut(V)$ is only
$3$-dimensional, this shows that in general two isomorphic smooth complete
intersections of type $(1,2)$ are not conjugate
under the action of $\Aut(V)$.

\begin{conjecture}
\label{conj:Mukai}
Let $g_1=7,8,9,10$, and $m=\dim(M_{g_1})$. For all $k\geq 1$,
two smooth  complete intersection curves of type $(1^{m-2},k)$ in
$M_{g_1}$
are projectively isomorphic if and only if they are conjugate under
the action of $G_{g_1}$.
For all $k\geq 2$,
two smooth complete intersection curves of type $(2,k)$ in $V \subset
\P^6$ 
are projectively isomorphic if and only if they are conjugate under
the action of $\Aut(V)$.
\end{conjecture}

\paragraph{Statement of results}
\label{p:mukai.results}
We prove a weak form of Conjecture~\ref{conj:Mukai}
for $g_1=7,8,9,10$,
see Corollary~\ref{c:isom-conj.k1} and Proposition~\ref{pr:conj.k2}.
If $k=1$, a direct analysis of the representations $U_{g_1}$ carried out in
\cite{dm} shows that a general canonical curve $C$ has only finitely
many models as a linear section of $M_{g_1}$ up to the action of $G$,
see \ref{p:apll-dm};
we show (Corollary~\ref{c:isom-conj.k1}) that there is in fact only
one such model if $C$ is general.\footnote
{After finishing this article, we have found out that for $g_1=7,8,9$
this had already been proven by Mukai himself
\cite{mukai-symm0}.}

If $k \geq 2$, we show
(Proposition~\ref{pr:conj.k2} and Remark~\ref{rk:mukai.larger.k})
that a general complete
intersection curve of type $(1^{m-2},k)$ has only finitely many models
on $M_{g_1}$ up to the action of $G$.
From this we can deduce
(Corollary~\ref{c:cork} and Remark~\ref{rk:mukai.larger.k})
that the moduli map
$c_g^k$ is birational, and a general complete intersection curve $C$
of type $(1^{m-2},k)$ on $M_g$ has $\cork(\Phi_C)=1$.
Thus the canonical model
of $C$ can be extended only once, to a $K3$ surface wich is the
reembedding by $k$-ics of a linear section of $M_{g_1}$.

On our way, we show that the Mukai variety in $\P(U_{g_1})$ is the
universal extension of its smooth curve linear sections,
see Proposition~\ref{p:mukai-univ}.
In particular it is not extendable, as was already observed in
\cite{CLM98}.

\paragraph{}
\label{p:g6.dimfibre}
When $g_1=6$, we cannot apply the same arguments because canonical
curves of genus $6$ do not satisfy property $N_2$.
However for $k=2$ the situation is pretty well understood.

A curve $C$ corresponding to a general point in the image of
$c_{21}^2$ is a complete intersection of 
type $(2,2)$ in $V$.
The canonical model of $C$ in $\P^{20}$
is a linear section of $V$ anticanonically embedded in $\P^{22}$.
This implies by \cite[Prop.~6]{mukai2} and
\cite[Prop.~5.4]{beauville-fano} that the fibre $(c_{21}^2)^{-1}(C)$
has dimension at least $1$.
In fact, it is proven in \cite{CLM98} that
$\cork(\Phi_C)= 2$, hence the dimension of the fibre is exactly
$1$.
Thus $V \subset \P^{22}$ is the universal extension of $C$; in
particular it is not extendable.

\subsection{Useful results}

In the statements below, we use the notation of \ref{p:mukai.descr}.

\begin{theorem}[\cite{dm}]
\label{stab-sous-espace}
Let $g=7,8,9,10$, 
and let $k_g=4,3,2,2$ for $g=7,8,9,10$, respectively.
Let $P \subset U_g$ be a generic vector subspace such that
\begin{equation}
\label{eq:c-dim.stab}
\tag{$\star$}
  \min\bigl(\codim_{U_g}(P),\dim(P)\bigr) > k_g.
\end{equation}
Then the stabilizer of $P$ under the action of $G$ is trivial.
\end{theorem}

\bigskip\noindent
(By the stabilizer of $P$ under the action of $G$, we mean the subgroup
of those $\gamma\in G$ such that $\gamma.P=P$; we do not ask that
$\gamma$ acts as the identity on $P$).
The situation when \eqref{eq:c-dim.stab} is an equality is analyzed in
detail in \cite{dm}.

Following Mukai's arguments in \cite{mukai1}, this implies that a
general linear section of the Mukai variety $M_g$ of positive
dimension and codimension $k>k_g$ has only trivial automorphisms, see
\cite[Thm.~1]{dm}.

\paragraph{}
\label{p:apll-dm}
Theorem~\ref{stab-sous-espace} implies that for a general canonical
curve $C$ of genus $g=7,8,9$
(resp.\ a general canonical of genus $10$ lying on a $K3$ surface), 
up to the action of $G$ there are finitely many
linear sections of $M_g$ isomorphic to $C$.
Consider the rational map
\[
  c_{M_g}:
  \G(g-1,\P(U_g))/G \dashrightarrow \M_g
\]
mapping the orbit of a general $(g-1)$-dimensional subspace
$\Lambda \subset \P(U_g)$
to the modulus of the curve $\Lambda \cap M_g$.
It is dominant if $g\neq 10$ by
\cite[Thm.~6.1]{mukai1},
and dominant onto a divisor in $\M_{10}$ if $g=10$
by \cite{cukierman-ulmer}.
By Theorem~\ref{stab-sous-espace}, the source has dimension
\[
  \dim \bigl(\G(g-1,\P(U_g)) \bigr)
  -\dim(G),
\]
and a direct computation shows that this equals
$\dim(\M_g)$ if $g\neq 10$, resp.\ $\dim(\M_{10})-1$ if $g=10$.
This implies that $c_{M_g}$ is generically finite, which proves
that for general $C$ there are finitely many linear sections of $M_g$
isomorphic to $C$, up to the action of $G$.

\paragraph{Lines through a point in a Mukai variety}
\label{p:VMRT}
The results in this paragraph follow from
\cite[Thm.~4.8]{landsberg-manivel}.
Let $g=7,8,9,10$, and $M_g$ be a Mukai variety.
For all $x\in M_g$, the variety of lines contained in $M_g$ and
passing through $x$ is itself a homogeneous variety which we shall
denote by $M'_g$. We give below a description of $M'_g$, but the only
fact we shall use is that $\dim(M'_g)=\dim(M_g)-4$.

\begin{center}
\begin{tabular}{l|llll}
\(g\) & \(M_g\) & \(\dim(M_g)\) & \(M'_g\) & \(\dim(M'_g)\)\\
\hline
\(7\) & \(S_{10}\) & \(10\) & \(\G(2,5)\) & 6\\
\(8\) & \(\G(2,6)\) & \(8\) & \(\P^1\times\P^3\) & 4\\
\(9\) & \(\LG(3,6)\) & \(6\) & \(\P^2\) & 2\\
\(10\) & \(G_2/P_2\) & \(5\) & \(\P^1\) & 1\\
\end{tabular}
\end{center}

In \cite[Thm.~4.8]{landsberg-manivel} there is a recipe in terms of Dynkin
diagrams to find $M'_g$. The homogeneous varities we consider are of
the form $G/P$, with $G$ a complex Lie group with semisimple Lie
algebra, and $P$ a maximal parabolic subgroup; the diagrams below 
encode the group $G$, and the coloured vertex encodes the subgroup
$P$. For these matters we refer to \cite[\SS 23.3]{FH}. 
The recipe for finding $M'_g$ is that one should cancel the coloured vertex, and colour
those vertices that were adjacent to the cancelled vertex.
\\[5mm]
$g=7$, $S_{10}$:
\includegraphics[align=c,scale=0.5]{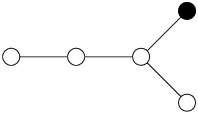}
\hspace{8mm} gives\hspace{4mm}
$\G(2,5)$:
\includegraphics[align=t,scale=.5]{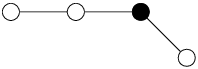}
\hspace{4mm}
in its Plücker embedding.
\\[5mm]
$g=8$, $\G(2,6)$:
\includegraphics[align=c,scale=.5]{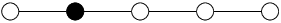}
\hspace{8mm} gives\hspace{4mm}
$\P^1\times \P^3$:
\includegraphics[align=c,scale=.5]{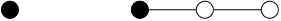}
\hspace{4mm}
as a Segre variety.
\\[5mm]
$g=9$, $\mathrm{LG}(3,6)$:
\includegraphics[align=c,scale=.5]{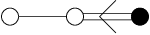}
\hspace{8mm} gives\hspace{4mm}
$\P^2$:
\includegraphics[align=c,scale=.5]{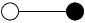}
\hspace{4mm}
embedded as a Veronese surface in $\P^5$, as the double edge
indicates. 
\\[5mm]
$g=10$, $G_2/P_2$:
\includegraphics[align=c,scale=.5]{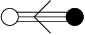}
\hspace{8mm} gives\hspace{4mm}
$\P^1$:
\includegraphics[align=c,scale=.5]{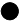}
\hspace{4mm}
embedded as a rational normal cubic, as the triple edge indicates,
contained in a hyperplane in $\P^4$ given by the contact structure on
the adjoint variety $G_2/P_2$.

There exist projective constructions of the Mukai varieties explaining
this correspondence in another way, which are instances of the general
construction given in \cite{landsberg-manivel2}.

\subsection{Results and proofs}

Let $g_1 \in \{7,8,9,10\}$.

\begin{proposition}
\label{p:mukai-univ}
Let $[C] \in \M_{g_1}$ be a general point if $g_1 \neq 10$, 
(resp.\ a general point in $\im (c_{10}^1)$ if $g_1=10$)
and consider $C$ in its
canonical embedding.
\subparagraph{}
\label{sp:all-integrable}
For all $e \in \ker(\trsp \Phi_C)$, the ribbon $C_e\subset \P^{g_1}$
is integrable to a unique surface (which is a $K3$ for general $e$).
\subparagraph{}
\label{sp:mukai-univ}
The Mukai variety $M_{g_1}$ is its universal extension.
\end{proposition}

\medskip
The generality condition on $[C]$ may be explicited as the requirement
that $C$ is a smooth linear section of the Mukai variety $M_{g_1}$,
see Corollary~\ref{c:mukai-univ+}.

\begin{proof}[Proof of \ref{sp:all-integrable}]
In this paragraph, we show the first part of the proposition.
If $g_1 \leq 9$, by \cite[Thm.~6.1]{mukai1}
the general curve of genus $g_1$ lies on a 
$K3$, so the map $c_{g_1}^1: \KC_{g_1} ^1 \to \M_{g_1}$ is dominant,
and its generic fibre has dimension $22-2g_1$;
on the other hand, we know that 
$\cork(\Phi_C)=23-2g_1$ by \cite[Cor.~4.4]{cm90}.
If $g_1=10$, by \cite{cukierman-ulmer}
$\im(c_{10}^1)$ is a divisor in $\M_{10}$,  and the general fibre of
$c_{10}^1$ has dimension $3 = \cork(\Phi_C)-1$.
The upshot is that in both cases there is a family of surface
extensions of $C$ of the same dimension as the moduli space
$\P(\ker(\trsp \Phi_C))$ of ribbons on $C$.

Moreover $C$ has Clifford index
$\lfloor \frac {g_1-1} 2 \rfloor >2$, so it satisfies property $N_2$.
It follows that for every $e \in \ker(\trsp \Phi_C)$ the corresponding
ribbon is integrable to at most one surface,
see Proposition~\ref{pr:unicity-integral}.
By the dimension count above, this implies that the general such
ribbon must indeed be integrable;
we shall see in the proof of \ref{sp:mukai-univ} that they are in fact
all integrable, which will end the proof of the first part of the
proposition.
\end{proof}

\begin{lemma}
\label{l:cone-mukai}
Let $\Lambda$ be a $(g_1-1)$-dimensional linear subspace of
$\P(U_{g_1})$ such that
$C= \Lambda \cap M_{g_1}$ is a smooth curve.
Then there does not exist any $g_1$-dimensional linear space $\tilde\Lambda
\subset \P(U_{g_1})$ containing $\Lambda$, and
such that $\tilde\Lambda \cap M_{g_1}$ is the cone over $C$ (with
vertex a point).
\end{lemma}

\begin{proof}
Let $v \in M=M_{g_1}$ be a point such that the cone $J(v,C)$ is
contained in $M$.
For all $x \in C$ the line $\vect {v,x}$ is contained in $M$
hence in the projective tangent space $\bT_{M,v}$, a
$\P^{m}$, $m=\dim(M)$. Since $C$ spans $\Lambda$, it follows that
$\tilde \Lambda = \vect{v,C}$ is a $\P^{g_1}$ contained in
$\bT_{M,v}$.
On the other hand the lines contained in $M$ and passing through
$v$ form a cone $K_v$ over $M' \subset \P(T_{M,v})$
(the homogeneous variety indicated in \ref{p:VMRT})
with vertex the point $v$.
Since $M'$ has dimension $m-4$,
the cone $K_v$ is $(m-3)$-dimensional contained in $\bT_{M,v}$.
Since also $\tilde \Lambda$ is contained in the
$m$-dimensional $\bT_{M,v}$,
the intersection $K_v \cap \tilde \Lambda$ must be at least
$(g_1-3)$-dimensional,
which is in contradiction with the fact that $\Lambda \cap M$ is
a curve. 
\end{proof}

\begin{proof}[Proof of \ref{sp:mukai-univ}]
Let $\Lambda$ be a general $\P^{g_1-1} \subset \P(U_{g_1})$, and
$C= \Lambda \cap M_{g_1}$. We denote by $\P(U_{g_1})/\Lambda$ the closed subset of 
${\bf G}(g_1, \P(U_{g_1}))$ of all linear subspaces of $\P(U_{g_1})$ of dimension $g_1$ containing $\Lambda$. We want to prove that the rational map
\[
  r:
  \tilde \Lambda \in \P(U_{g_1})/\Lambda
  \dasharrow  2C_{\tilde \Lambda \cap M_{g_1}} \in \P(\ker(\trsp \Phi_C))
\]
is an isomorphism (this is the map which sends $\tilde\Lambda$ 
to the ribbon of $C$ in the surface
$\tilde\Lambda \cap M_{g_1}$).

This is a projectivity (i.e., it is induced by a linear map at vector space level) between projective spaces of the same dimension
$n(g_1)-g_1$
(remember that $n(g_1)=\dim(\P(U_{g_1}))$, and
the values of $\cork(\Phi_C)$ are 
given in the proof of \ref{sp:all-integrable}).
Therefore it suffices to prove that
the associated linear map between vector spaces has trivial 
kernel. Since ${\rm Cliff}(C)>2$ the trivial ribbon over $C$ may only be integrated
to the cone over $C$ (Proposition~\ref{pr:unicity-integral}).
By Lemma~\ref{l:cone-mukai} there is no $\tilde \Lambda$ cutting out
the cone over $C$ in $M_{g_1}$, so our map $r$ is indeed injective, hence
an isomorphism.
\end{proof}

\begin{corollary}
\label{c:mukai-univ+}
Let $C$ be a smooth curve linear section of $M_{g_1}$.
Then \ref{sp:all-integrable} and \ref{sp:mukai-univ} hold for $C$.
In particular, the corank of the Gauss--Wahl map $\Phi_C$ is the same
for all smooth curve linear sections of $M_{g_1}$.
\end{corollary}

\begin{proof}
By Lemma~\ref{l:mukai-cliff} below, one has $\Cliff(C)>2$.
Therefore, Lemma~\ref{l:cone-mukai} and the proof
of \ref{sp:mukai-univ} show that the map
\[
  r:
  \tilde \Lambda \in \P(U_{g_1})/\Lambda
  \to  2C_{\tilde \Lambda \cap M_{g_1}} \in \P(\ker(\trsp \Phi_C))
\]
(in the same notation as in the
proof of \ref{sp:mukai-univ})
is an injective linear map of projective spaces.
This implies that the universal extension of $C$ is an extension of
$M_{g_1}$. But $M_{g_1}$ is not extendable
(because for a general $C'$ curve linear section of $M_{g_1}$,
$M_{g_1}$ is the universal extension of $C'$ by what we have already
proved), 
therefore the Mukai
variety is also the universal extension of $C$, and the corank of the
Gauss--Wahl map of $C$ equals that of a general curve linear section 
of $M_{g_1}$.
\end{proof}

\begin{lemma}
\label{l:mukai-cliff}
Let $g_1=7,8,9,10$ and $C$ be a smooth linear section of the Mukai
variety $M_{g_1}$. Then one has $\Cliff(C)>2$.
\end{lemma}

\begin{proof}
For $g_1=7,8,9$, Mukai
\cite{mukai-symm0,mukai-symm1,mukai-symm2}
has proven that a curve of genus $g_1$ is a
linear section of $M_{g_1}$ if and only if it has
no $g^1_5$,
resp.\ no $g^2_7$,
resp.\ no $g^1_6$.
On the other hand a curve of genus $g_1<10$ has Clifford index
strictly larger than $2$ if and only it has no $g^1_4$. So for
$g_1=7,9$, any smooth curve linear section of $M_{g_1}$ has Clifford index
strictly larger than $2$.

For $g_1=8$ we use a different argument.
By \cite{josefiak-pragacz} the Mukai variety $M_8=\G(2,6)$ satisfies
the property 
$N_2$. Then we find by applying repeatedly Green's hyperplane sections
theorem \cite[Thm.3.b.7]{green84} that $C$ enjoys property $N_2$ as well.
In turn this implies by the Green--Lazarsfeld theorem
\cite[Appendix]{green84} that $\Cliff(C)>2$.

For $g_1=10$, it is proven in \cite[Rmk.~2.7]{cukierman-ulmer} that
any smooth curve linear section of $M_{10}$ has Clifford index
strictly larger than $2$.
\end{proof}

\begin{corollary}
\label{c:isom-conj.k1}
Two smooth curve linear sections of $M_{g_1}$ are isomorphic if
and only if they are conjugate under the action of $G$.
\end{corollary}

\begin{proof}
Let $C$ and $C'$ be two distinct general curve linear sections of
$M_{g_1}$, and assume that they are isomorphic
(being canonical curves, they are isomorphic as abstract curves if and
only if they are isomorphic as polarized varieties).

By \ref{sp:mukai-univ} (Corollary \ref{c:mukai-univ+})
we may choose two $K3$ surfaces $S$
and $S'$ sections of $M_{g_1}$ containing $C$ and $C'$ respectively,
and such that the two ribbons $2C_S$ and $2C'_{S'}$ are isomorphic.
By unicity of the integral of this ribbon, there exists an isomorphism
of polarized surfaces $\phi:(S,C) \cong (S',C')$ taking $C$ to $C'$.
By \cite[Thm.~0.2]{mukai1} there exists $\gamma \in G$
inducing $\phi$, and in particular $\gamma.C=C'$.
\end{proof}

\begin{proposition}
\label{pr:conj.k2}
Let $C$ be a general complete intersection of type $(1^\nu,2)$ in
$M_{g_1}$, $\nu=n(g_1)-g_1$. 
There are at most finitely many curves $C'$ complete intersection of
the same kind, projectively isomorphic to $C$ but not conjugate to it
modulo $G$. 
\end{proposition}

\begin{proof}
We argue by contradiction and assume that for general $C$ as above
there exists a positive dimensional family of curves $C'$
projectively isomorphic to it but not $G$-conjugate.

We claim that a fortiori the same holds for all (even singular) complete intersection curves of type  
$(1^{\nu},2)$ in $M=M_{g_1}$.
Indeed, as $C$ moves in the family of such complete intersections,
the dimension of the family of curves in $M$
projectively isomorphic to $C$
(resp.\ of the stabilizer of $C$ in $G$) is upper semi-continuous, so
that the family of curves isomorphic to $C$ gets bigger whereas that of
curves conjugate to $C$ gets smaller.
The second of these semi-continuity statements follows from the fact
that there is a universal family of stabilizers over the family of
complete intersection curves in $M$, which we shall denote by
$\IC(1^\nu,2;M)$.
For the former, we have to consider the rational map from complete
intersections in $M$ to the moduli space $\M_g$, $g=1+4(g_1-1)$;
one may resolve its indeterminacy locus by suitably blowing-up
$\IC(1^\nu,2;M)$;
the obtained morphism gives us the
semi-continuity we want.

In particular the assumption we made by contradiction holds for
$C=2C_1$, a very general ribbon over a  
curve linear section $C_1$ of $M$,
which is indeed a complete intersection of type $(1^\nu,2)$ in $M$
for which the quadratic equation is a square.
Our contradiction assumption implies that there exists another ribbon
$C'$, projectively isomorphic but not $G$-conjugate to $C$
(we need to assume that the general complete intersection $C$ has
infinitely many projectively isomorphic but not conjugate copies to
reach this conclusion, to avoid that finitely many such copies all
degenerate to the same one with multiplicity, as $C$ degenerates to a
ribbon).

The ribbon $C'$ is incarnated on a copy 
$C_1'$ of $C_1$, which is a curve linear section of $M$ as well.
Both $C$ and $C'$ are integrable to surfaces $S$ and $S'$
respectively, both linear sections of $M$.
By the unicity of the extension (see \ref{pr:unicity-integral}), there is a projectivity $\omega$ such that $\omega(S)=S'$ and $\omega(C)=C'$. By  \cite[Thm.~0.2]{mukai1}, $S$ and $S'$ are conjugate under the action of $G$, so there exists $\gamma \in G$ such that $\gamma.S=S'$.
The very generality of $C$ implies that of $S$, so we may assume that
$S$ has no non-trivial projective automorphisms
by Proposition~\ref{pr:no-auto-vgK3}. This implies that $\omega=\gamma$, hence $\gamma.C=C'$, a contradiction.
\end{proof}

\begin{remark}
The rational map from $\IC(1^\nu,2;M)$ to $\M_g$ considered in the
proof is indeterminate on the locus corresponding to ribbons $2C_1$,
but this indeterminacy may be generically resolved at a general point of the ribbons locus by a single blow-up
along this locus: a general point in the exceptional divisor lying
over a ribbon $2C_1$ is mapped to a double cover of $C_1$ branched
over a bicanonical divisor, as one sees by stable reduction. 
\end{remark}

\begin{corollary}
\label{c:cork}
Let $k=2$ and $g=4g_1-3$.
For general $[C] \in \im(c_g^2)$, one has
\[
  \dim\bigl( (c_{g} ^2)^{-1}(C) \bigr)
  = \cork(\Phi_C) -1
  = 0.
\]
\vspace{-5mm}
\end{corollary}

\begin{proof}
Being a general member of the image of $c_{g} ^2$,
the curve $C$ has a model as a complete intersection of
type $(1^{\nu},2)$ in $M_{g_1}$.
Such a curve spans a $\P^{g_1}$ in $\P(U_{g_1})$,
so there is a unique surface $S$ linear section of $M_{g_1}$ containing
it.
For $\gamma \in G$, the surface spanned by $\gamma.C$ is $\gamma.S$,
and the two pairs $(S,C)$ and $(\gamma.S,\gamma.C)$ give the same
point in the fibre of $c^2_g$ over $C$.

By Proposition~\ref{pr:conj.k2} $C$ has finitely many models
as a complete intersection in $M_{g_1}$,
up to the action of $G$.
Therefore the fibre $(c_{g} ^2)^{-1}(C)$ has dimension $0$.
By Proposition~\ref{pr:condCDS}, the results of \cite{cds} apply to
$C$, hence $\dim\bigl( (c_{g} ^2)^{-1}(C) \bigr)
= \cork(\Phi_C) -1$
(and in fact the fibre consists of a single point).
\end{proof}

\begin{corollary}
\label{c:stab}
The stabilizer of a general complete intersection curve of type
$(1^\nu,2)$ in $M_{g_1}$ under the action of $G$ is finite.
\end{corollary}

\begin{proof}
There is a dominant rational map
\[
  \IC(1^\nu,2;M)/G \dashrightarrow \im(c_g^2),
\]
and by Proposition~\ref{pr:conj.k2} it is generically finite.
Therefore its source and target have the same dimension.
By Corollary~\ref{c:cork} the dimension of the target equals that of 
$\KC_g^2$. A direct computation shows that
\[
  \dim(\KC_g^2) = \dim \bigl( \IC(1^\nu,2;M)\bigr)
  -\dim(G),
\]
and the assertion follows.
\end{proof}

\begin{remark}
\label{rk:mukai.larger.k}
The statements \ref{pr:conj.k2},
\ref{c:cork}, and
\ref{c:stab}
generalize mutatis mutandis to general complete intersection curves of
type $(1^\nu,k)$ for all $k\geq 2$.
To do so, one ought to replace in the proofs the ribbons $2C_1$ by
``higher order ribbons'' $kC_1$, cut out by $\nu$ linear equations and
one $k$-th power 
of a linear equation, and note that $2C_1 \subset kC_1$.
\end{remark}

\bigskip
We conclude by noting that we
cannot directly reproduce the argument of
Corollary~\ref{c:isom-conj.k1} to prove that
<<~two very general complete intersection curves of type
$(1^\nu,2)$ in $M_{g_1}$
are isomorphic if and only if they are conjugate under the action of
$G$~>>, 
because of the possibility that two
non isomorphic curves $C$ and $C'$ in $\P^{g_1}$ may give the same
canonical curve after reimmersion in $\P^g$; in other words 
the problem is that the canonical class may have distinct square roots.

\subsection{Maximal variation}

\begin{proposition}
\label{p:maxvar}
Let $g_1=7,8,9$ (resp.\ $g_1=10$), and $C$ be a general genus $g_1$
curve (resp. a curve of genus $10$ general among those that lie on a
$K3$ surface).
Let $S$ be a polarized $K3$ surface having $C$ as a hyperplane
section.
There are only finitely many members $C' \in |\O_S(C)|$ that are
isomorphic to $C$.
\end{proposition}

\medskip
As in Proposition~\ref{p:mukai-univ}, the generality condition on $C$
may be replaced by the condition that $C$ is a smooth linear section
of the Mukai variety $M_{g_1}$, see Corollary~\ref{c:mukai-univ+}.

\begin{proof}
By Proposition~\ref{p:mukai-univ} there exists a 
universal family of surface extensions of $C$ defined over
$\P\bigl( \ker(\trsp \Phi_C) \bigr)$
and a rational map
\(
  s: \P\bigl( \ker(\trsp \Phi_C) \bigr)
  \dashrightarrow
  \Kcan_{g_1}
\),
which sends a (non-trivial) ribbon over $C$ to the modulus of its
unique $K3$ integral.
We may thus apply \cite[Prop.~8.4]{cds} and conclude that
$\restr s U$ is finite, with $U$ the dense open subset on which $s$ is
well-defined. In particular the ribbons $2C'_S$ of the various copies
$C'$ of $C$ in $|\O_S(C)|$ are only finitely many
(in other words: the first infinitesimal
neighbourhoods of $C'$ in $S$ fall into finitely many isomorphism
classes).

On the other hand, arguing exactly as in
\cite[Cor.~8.6]{cds} we conclude 
that for all ribbon $C_e$ over $C$ the copies $C'$ of $C$ in
$|\O_S(C)|$ such that $2C'_S=C_e$ are finitely many, and this ends the
proof. 

Before we close this proof, we emphasize that we have all the
necessary assumptions for the arguments of \cite[Prop.~8.4]{cds}
and \cite[Cor.~8.6]{cds} to apply without any single change.
There the assumption that ``$g \geq 11$ and
$\Cliff(C)>2$'' is made only to ensure that every ribbon over $C$ is
integrable to a unique surface, and this in the present situation is
granted by Proposition~\ref{p:mukai-univ}.
\end{proof}

\begin{remark}
In fact, if $S$ is general then $C$ and $C'$ have the same ribbon in
$S$ only if they are the same curve.
Indeed if $C$ and $C'$ have the same ribbon, then by
Proposition~\ref{pr:unicity-integral} there exists a projective
automorphism of $S$ mapping $C$ to $C'$.
By \cite[Thm.~1]{dm} the automorphism group of $S$ is trivial, hence
$C=C'$. 
\end{remark}


\begin{corollary}
Let $(S,L)$ be a general primitively polarized $K3$ surface of genus
$g_1=7,8,9,10$. Then for general $C \in |L|$, there are only finitely
many members $C' \in |L|$ such that $C$ and $C'$ are isomorphic.
\end{corollary}

\begin{proof}
The generality assumptions ensure that $C$ is liable to
Proposition~\ref{p:maxvar}.
\end{proof}

\paragraph{Problem}
If $C$ is a general curve of genus $g_1 \leq 6$, then it is not true
that the integral of a ribbon over $C$ is unique, so
both the
arguments given in \cite{cds} to prove
\cite[Prop.~8.4]{cds} and \cite[Cor.~8.6]{cds} break down, and it is
not clear 
to us whether Proposition~\ref{p:maxvar} holds in this case.


\section{Theta-characteristics, spin curves, etc.}
\label{S:spin}

In this section we will discuss some properties of $\S_g^{\frac 1 k, g_1}$, of $\T_g^{\frac 1 k, g_1}$ and of $\im(c _g ^k)$. 

\paragraph{Sextic double planes}\label{par:6ic} First we consider $K3$ surfaces which are double cover of the projective plane branched along a smooth sextic curve. In relation with these surfaces we can consider $\K_5^2$. The image of $c_5^2$ is the hyperelliptic locus in $\M_5$,
and this equals $\T _5 ^{\frac 1 2, 2}$. Indeed if $(C,L) \in \S _5
^{\frac 1 2, 2}$, $|L|$ gives a special $g^2_4$ on $C$ and by Clifford's Theorem $C$ is hyperelliptic.  

The image of $c_{10}^3$ is  the locus of plane sextics in $\M_{10}$, which
coincides with $\T _{10} ^{\frac 1 3, 2}$. In fact a cubic root $\theta$ of the canonical bundle $K$ with $h^0(\theta)=3$ is such that $|\theta|=g^2_6$.

For $k \geq 4$ it is more complicated to understand $\S_{1+k^2}^{\frac 1 k, 2}$, $\T_{1+k^2}^{\frac 1 k, 2}$ and whether the equality holds in \eqref {eq:r-spin-img}. 

\paragraph{Complete intersections}
\label{p:spin-CI}
Next we consider the cases $g_1=3,4,5$, in which the general polarised
$K3$ surface $(S,L_1)$ with $L_1$ a primitive line bundle with
$(L_1)^2=2g_1-2$ is a complete intersection. In this case we face the
following question.

\medskip
\subparagraph{Question}
\label{q:Sg} Let $g_1=3,4,5$.
Is it true that for any $k\geq 2$ the moduli space $\S_g ^{\frac 1 k, g_1}$ is irreducible and that the general member of \smash{$\S_g ^{\frac 1 k, g_1}$}
is a complete 
intersection of type $(4,k)$ in $\P^3$ if $g_1=3$, of type 
$(2,3,k)$ in $\P^4$ if $g_1=4$, and of type $(2,2,2,k)$ in $\P^5$ if $g_1=5$?

\medskip
It is worth noticing that for $k=2$ we have 
\begin{center}
\begin{tabular}{r|lll}
$g_1$ & 3 & 4 & 5 \\
\hline
$\dim (\KC_g ^2)-\expdim(\T _g ^{\frac 1 2, g_1})$
& $10$ & $6$ & $3$
\\
dim(fibre of $c_g^2$)
& $10$ & $6$ & $3$
\end{tabular}
\end{center}
where the values on the first line are computed with
\eqref{expdim-spin},
and those on the second line are obtained as in Section~\ref{S:CI} for $g_1=3$.
\medskip

In the rest of this section we discuss Question~\ref{q:Sg}. First we give a definition. We say that an irreducible component $\S$ of
  $\S _g ^{\frac 1 k, g_1}$
  is \emph{birational} if for $(C,\theta) \in \S$ general the linear
  series $|\theta|$ determines a birational map.
There are three levels of problems related to  Question~\ref{q:Sg}:\\
\begin{inparaenum}
\item [(a)] prove that $\S _g ^{\frac 1 k, g_1}$ is irreducible, and its general
element is a complete intersection of type $(4,k)$ in $\P^3$ if $g_1=3$, of type 
$(2,3,k)$ in $\P^4$ if $g_1=4$, and of type $(2,2,2,k)$ in $\P^5$ if $g_1=5$;\\
\item [(b)]  prove that $\S _g ^{\frac 1 k, g_1}$ has only one birational
irreducible component and its general element is a complete
intersection of type $(4,k)$ in $\P^3$ if $g_1=3$, of type 
$(2,3,k)$ in $\P^4$ if $g_1=4$, and of type $(2,2,2,k)$ in $\P^5$ if $g_1=5$;\\
\item [(c)]  prove that the closure of the family of complete intersections of type $(4,k)$ in $\P^3$ if $g_1=3$, of type 
$(2,3,k)$ in $\P^4$ if $g_1=4$, and of type $(2,2,2,k)$ in $\P^5$ if $g_1=5$, is an irreducible component of $\S _g ^{\frac 1 k, g_1}$.
\end{inparaenum}

\medskip
A  justification
for restricting our attention to
the birational components is that there are in general
several non-birational components, and we do not care entering the
corresponding botanic. 
We provide examples in
\ref{p:nonbirat.hyperell},
\ref{p:nonbirat.biell} and
\ref{p:nonbirat.g2}
below.

Also, one could try to characterize the birational components by some
Brill--Noether theoretic property. We will see examples of this in Propositions \ref {prop:g1=5}, \ref  {prop:no}, \ref {prop:bir} and \ref {prop:g_1=4,k=3}. 
\medskip

We now answer affirmatively to problem (c) above:

\begin{proposition}\label{prop:compint} For $3\leq g_1\leq 5$ and for any integer $k\geq 2$ the closure of the family of complete intersections of type $(4,k)$ in $\P^3$ if $g_1=3$, of type 
$(2,3,k)$ in $\P^4$ if $g_1=4$, and of type $(2,2,2,k)$ in $\P^5$ if $g_1=5$, is an irreducible component of $\S _g ^{\frac 1 k, g_1}$. 
\end{proposition}

\begin{proof} Let $\S$ be an irreducible component of $\S _g ^{\frac 1 k, g_1}$ containing the family of complete intersections of type $(4,k)$ in $\P^3$ if $g_1=3$, of type 
$(2,3,k)$ in $\P^4$ if $g_1=4$, and of type $(2,2,2,k)$ in $\P^5$ if $g_1=5$. We have to prove that if $(C,\theta)\in \S$ is general, it corresponds to a complete intersection of the same type in $\P^{g_1}$. Note that the linear series $|\theta|$ has dimension $g_1$, it is very ample, and maps $C$ to a smooth curve in $\P^{g_1}$. Since a complete intersection is projectively normal, then also $C$ is projectively normal in $\P^{g_1}$. Moreover $k\theta=K_C$. By Gherardelli's Theorem (see \cite[p. 147]{acgh}), this settles the case $g_1=3$.

Next we treat the case $g_1=4$. Let $C'$ be a smooth complete
intersection of type $(2,3,k)$ in $\P^4$. For any positive integer $h$
we have, by semicontinuity and by projective normality,
\[
h^0(\mathcal O_{\P^4}(h))-h^0(\mathcal I_C(h))=h^0(\mathcal O_C(h))\leq  h^0(\mathcal O_{C'}(h))=h^0(\mathcal O_{\P^4}(h))-h^0(\mathcal I_{C'}(h))
\]
hence $h^0(\mathcal I_C(h))\geq h^0(\mathcal I_{C'}(h))$. On the other
hand, by semicontinuity, we have $h^0(\mathcal I_C(h))\leq
h^0(\mathcal I_{C'}(h))$ so that equality holds. This immediately
implies that $C$, as well as $C'$, is a complete intersection of type
$(2,3,k)$.

The case $g_1=5$ is similar and we mostly leave it to the reader. In
this case we find a net of quadrics containing $C$. The base locus of
this net is a complete intersection surface of type $(2^3)$ because it
so happens for complete intersection curves, of which $C$ is a
deformation. 
\end{proof}

\bigskip
It is useful, in order to compute the number of moduli of non-birational components of $\S_g^{\frac 1k,g_1}$, to record the following well known lemma.

\begin{lemma}
\label{l:dim-loci}
The following loci in $\M_g$ have the indicated dimensions:
\begin{mycpctitem}
\item $k$-gonal locus $\M_{g,k}^1$: $2g+2k-5$;
\item locus of $k$-elliptic curves, i.e., $k:1$ covers of elliptic curves: $2g-2$;
\item $k:1$ covers of a genus $h>1$ curve: $2g-(2k-3)h+2k-5$.
\end{mycpctitem}
\end{lemma}

\begin{proof}
A $k:1$ cover of genus $g$ of $\P^1$ amounts to the datum of 
$2g-2+2k$ branch points, and these have $2g-2+2k-3$ moduli.

A $k:1$ cover of genus $g$ of an elliptic curve $E$ amounts to the
datum of $2g-2$ branch points, and these have $2g-2-1$ moduli. There
is one additional modulus for the choice of $E$.

A $k:1$ cover of genus $g$ of a genus $h$ curve $C$ amounts to the
datum of $2g-2-k(2h-2)$ branch points, and each points gives a
modulus. To this we have to add the dimension $3h-3$ of $\M_h$.
\end{proof}

\paragraph{Hyperelliptic curves}
\label{p:nonbirat.hyperell}
Let $(C,\fl)$ be a {genus $g$ hyperelliptic curve}. By this we mean that
$\fl$ is the $g^1_2$ on $C$. Then $K_C=(g-1)\fl$. If
\(
  g \congru 1 \mod r,
\)
then $\theta = \frac {g-1}r \fl$ is such that $r\theta=K_C$, and
one has 
$h^0(\theta) = \frac {g-1}r+1$.
We thus get a locus
\[
  \mathcal H_{g,\frac {g-1}r}
  \subset
  \S _g ^{\frac 1r, \frac {g-1}r}
\]
of dimension $2g-1$.
By Clifford's Theorem, all points in $\mathcal H_{g,\frac {g-1}r}$
correspond to hyperelliptic curves.  

More generally,
assume there exist non-negative integers $a,b,h$ such that
$b(2a+h)=g-1$, and $h \leq 2g+2$.
Then we can choose $h$ distinct Weierstrass points
$p_1,\ldots,p_h$ of $C$, and consider the linear series
$\theta=a\fl+p_1+\cdots+p_h$. Then $2b\theta =K_C$ and
$h^0(\theta)=a+1$.  In this way we obtain a locus
\[
  \mathcal H_{g,a}^h
  \subset
  \S _g ^{\frac 1{2b}, a}
  = \S _g ^{\frac {2a+h}{2g-2}, a}
\]
of dimension $2g-1$.

\paragraph{Bielliptic curves}
\label{p:nonbirat.biell} 
Let us fix an integer $k\geq 2$ and an integer $g_1\geq 1+\frac
2{k-1}$. Set, as usual, $g=1+k^2(g_1-1)$.  A bielliptic canonical
curve $C$ of genus $g$ in $\P^{g-1}$ sits on a cone $S$ with vertex a
point $p$ over an elliptic normal curve $E$ of degree $g-1$ spanning a
hyperplane of $\P^{g-1}$, and it is cut out on $S$ by a quadric
hypersurface not passing through $p$. The bielliptic involution
$\gamma^1_2$ is cut out on $C$ by the rulings of $S$. The linear
series $g_{2g-2}^{g-2}$ cut out on $C$ by the hyperplanes through $p$
is composed with the bielliptic involution $\gamma^1_2$, being the
pull-back on $C$ of the complete hyperplane series $g^{g-2}_{g-1}$ on
$E$. The latter series is certainly divisible by $k$, any of its
$k$-tuple
divisors is a $g_{k(g_1-1)}^{k(g_1-1)-1}$ on $E$, its pull-back
$\theta$ on $C$ is such that $(C,\theta)\in \mathcal
S_g^{\frac 1k, k(g_1-1)-1}$, and $k(g_1-1)-1\geq g_1$
by the hypothesis $g_1\geq 1+\frac 2{k-1}$. In this way we get a locus
\[
  \mathcal E_{g,k}
  \subset
  \mathcal S_g^{\frac 1k, k(g_1-1)-1}
\]
of dimension $2g-2$.
As we will see later in the case $k=2$,
$\mathcal E_{g,k}$ sometimes is a component of $\mathcal
S_g^{\frac 1k, k(g_1-1)-1}$, and sometimes not.

\paragraph{Double covers of genus $2$ curves}
\label{p:nonbirat.g2}

 Again we fix integers $k\geq 2$ and $g_1\geq 1+\frac 3{k-1}$ and set
$g=1+k^2(g_1-1)$. Fix in $\P^g$ a line $L$ and a subspace $\Pi$ of
dimension $g-3$ skew with $L$. Consider in $\Pi$ a linearly normal,
smooth genus two curve $\Gamma$ of degree $g-1$. Consider the $g^1_2$
on $\Gamma$ as a $2:1$ morphism $f:\Gamma\to L$. For each point $p\in
\Gamma$ consider the line $\langle p,f(p)\rangle$, and take the union
$S$ of all such lines when $p$ varies in $\Gamma$. Then $S$ is a
scroll of genus $2$ and degree $g+1$, having $L$ as a double directix
line. Take the intersection curve $C$ of $S$ with a quadric
intersecting $L$ in two general points $q_1,q_2$, off the four rulings
issuing from $q_1,q_2$. The curve $C$ is smooth, does not intersect
$L$, and it is easy to see that it has genus $g$, hence it is a
canonical curve with a $2:1$ morphism onto the genus 2 curve $\Gamma$,
and the genus 2 involution $\gamma^1_2$ is cut out on $C$ by the
rulings of $S$. The base point free linear series $g_{2g-2}^{g-3}$ cut
out on $C$ by the hyperplanes containing $L$ is composed with the genus 2
involution $\gamma^1_2$, being the pull-back on $C$ of the complete
hyperplane series $g^{g-3}_{g-1}$ on $\Gamma$. This series is
certainly divisible by $k$, any of its $k$-tuple divisors is a
$g_{k(g_1-1)}^{k(g_1-1)-2}$ on $\Gamma$, its pull back $\theta$ on $C$
is such that $(C,\theta)\in \mathcal S_g^{\frac 1k,
k(g_1-1)-2}$, and $k(g_1-1)-2\geq g_1$. In this way we get a locus
\[
  \mathcal D_{g,k}
  \subset
  \mathcal S_g^{\frac 1k, k(g_1-1)-2}
\]
of dimension $2g-3$. 

\paragraph{Curves on quadrics in $\P^3$}
\label{p:g3.quadric} 

Fix an integer $k\geq 3$. 
We shall now consider irreducible curves of degree $4k$ and geometric
genus $1+2k^2$ that lie on a
quadric in $\P^3$, and study the possibility that the pull-back of the
hyperplane bundle is a $k$-th root of the canonical bundle on the
normalization.

We look at the case $k=3$ and we consider irreducible curves of type $(6,6)$ on a smooth quadric  $S=\P^1 \times \P^1$.  These curves must have exactly 6 nodes, or equivalent singularities, to have  genus $19$. Let us assume they have only nodes. For such a curve $\Gamma$ the adjoint system 
$K_S+\Gamma$ has bidegree $(4,4)$, so we find a $3$-rd root of the canonical bundle on the
normalization $C$ when the $6$ nodes lie on a plane conic section $D$
of $S$, and are the complete intersection of $\Gamma$ and $D$. Let us
check that the curves in question exist indeed.

Fix an irreducible  conic $D$ on $S$ and $6$ general points $p_1,\ldots, p_6$ on it. Consider the linear system $\mathcal L$ of curves of type $(6,6)$ singular at $p_1,\ldots, p_6$. One has
\[
\dim(\mathcal L)\geq 48-3\cdot 6=30.
\]
Note  that $\mathcal L$ contains the subsystem $\mathcal L'$ consisting of the curves containing $D$ as a fixed component, with variable part consisting of curves of type $(5,5)$ containing $p_1,\ldots, p_6$. One has
\[
\dim (\mathcal L')=35-6=29
\]
(the points $p_1,\ldots, p_6$ clearly impose independent conditions to curves of type $(5,5)$) and the general curve in $\mathcal L'$ has nodes at $p_1,\ldots, p_6$. Since $\mathcal L'$ is strictly contained in $\mathcal L$, we see that the general curve in $\mathcal L$ is irreducible, has nodes at $p_1,\ldots, p_6$ and no other singularity, so it is of the required type. Note that in fact $\dim(\mathcal L)=30$. Indeed, with only one condition $D$ splits off the curves of $\mathcal L$ and the residual system is just $\mathcal L'$, which therefore has codimension 1 in $\mathcal L$. 

This construction gives us a locus $\S$ inside $\S_{19}^{\frac 13,3}$. Let us compute its dimension. The choice of the plane section $D$ of $S$ depends on 3 parameters. The choice of $p_1,\ldots, p_6$ on $D$ depends on 6 parameters. The linear system $\mathcal L$ has dimension 30. The automorphisms of $\P^1\times \P^1$ have $6$ dimensions. In conclusion we find
$\dim(\S)=33$. Note that one can make the same construction on a quadric cone, thus getting curves in the closure of $\S$. 

It is interesting to compare the dimension of $\S$ with the dimension of the image of $c_{19}^3$, which is $34$ (see Section \ref {S:CI}). Since $\dim(\S)<\dim(\im (c_{19}^3))$, 
this suggests the following conjecture.

\medskip
\subparagraph{Conjecture}
\label{conj:gen} 
The locus $\S$ is contained in  the closure of $\im (c_{19}^3)$.
\medskip

It would be tempting to see the above example as a particular case of a more general situation. Namely, for every integer $k\geq 3$ one would like to consider irreducible  curves of type $(2k,2k)$ on $\P^1 \times
\P^1$, with exactly $2k(k-2)$ nodes, so that their genus is
$1+2k^2$. For such a curve $\Gamma$, in order to have a $k$-th root
of the canonical bundle on the normalization $C$, the $2k(k-2)$ nodes
should lie on a curve $D$ of type $(k-2,k-2)$, and should be the
complete of intersection $\Gamma$ and $D$. However,  as soon as $k\geq
4$, it is not at all clear that irreducible curves of the required
type exist indeed. 
It is also possible to consider similar examples on rational normal
scrolls in $\P^4$ and $\P^5$ respectively; these pose the same kind of
questions, which we don't answer in this text.

\subsection{Theta-characteristics (the case $k=2$)}
\label{s:theta-details}

The case of theta-characteristics is special in that we have an
expected dimension for
$\S _g ^{\frac 1 2, g_1}$ or equivalently for $\T _g ^{\frac 1 2, g_1}$.

\paragraph{$g_1=3$ and $k=2$}
\label{l:S_9^3}

Consider the locus $\mathcal H_{9,4}\subseteq \S _9 ^{\frac
12,4}\subseteq \S _9 ^{\frac 12,3}$ introduced in paragraph~\ref
{p:nonbirat.hyperell}. The elements in $\mathcal H_{9,4}$ are of the
type $(C,4\fl)$, where $(C,\fl)$ is a genus $9$ hyperelliptic
curve. One has $h^0(4\fl) = 5$ (odd),
hence $(C,4 \fl)$ cannot be a specialization of some $(C_1,\theta)\in \S _9
^{\frac 1 2, 3}$ with $h^0(\theta)=4$ (even) by
Theorem~\ref{t:harris} (the invariance of the parity is due to
Mumford \cite{mumford}). 
Note that $\dim(\mathcal
H_{9,4})=17$, and this agrees with Theorem~\ref{t:harris} because
\[
17=\dim(\mathcal H_{9,4})\geq \dim(\M_9) - \frac {4(4+1)} 2=14.
\]

Let $M$ be any irreducible component of $\S _9 ^{\frac 12,3}$ whose general element $(C,\theta)$ is such that $h^0(\theta)$ is even. By Clifford's Theorem, if $(C,\theta)$ is general in $M$, then $h^0(\theta)= 4$.  Moreover, by Theorem~\ref{t:harris} one has
\begin{equation}\label{eq:m}
\textstyle
\dim(M) \geq 
 \dim(\M_9) - \frac {3(3+1)} 2
= 18.
\end{equation}

Note that the locus $\mathcal H_{9,3}^2$ consisting of hyperelliptic curves (see 
paragraph \ref {p:nonbirat.hyperell}), does not fill up an irreducible
component of $\S _9 ^{\frac 12,3}$ because it has only dimension $17$
and it is not contained in $\mathcal H_{9,4}$ by the constancy of $h^0(\theta)$.

\begin{theorem}\label{thm:g3}
$\S _9 ^{\frac 1 2, 3}$ has two irreducible components, one
equal to $\mathcal H_{9,4}$ and the other whose general
element is
a complete intersection curve of type $(2,4)$ in $\P^3$ (see Proposition \ref {prop:compint}).
\end{theorem}

\begin{proof}
Let $M$ be an irreducible component of $\S _9 ^{\frac 12,3}$
containing $\mathcal H_{9,4}$. If $(C,\theta)$ is general in $M$, by
constancy of the parity one has $h^0(\theta)=5$, so $M$ coincides with $\mathcal H_{9,4}$ by Clifford's Theorem. 

Let now $M$ be an irreducible component of $\S _9 ^{\frac 12,3}$ whose general element 
$(C,\theta)$ has $h^0(\theta)=4$, so $|\theta|$ is a $g^3 _8$.
 By \eqref {eq:m}, $C$ is not hyperelliptic, since $\dim(\M_{9,2}^1)=17$. Let us show that $\theta$ determines an embedding 
of $C$ in $\P^3$.

Let $b$ be the number of
base points of $|\theta|$. Then $|\theta|$ induces a special $g^3
_{8-b}$ hence $b<2$ by Clifford's Theorem. 
If $b=1$, then the $g^3_7$ must give a birational map since $7$ is
prime. This is impossible because by Castelnuovo's bound a curve of genus $7$ 
in $\P^3$ has at most genus 6. Hence $|\theta|$ is base-point-free.

Now assume that the morphism $C \to \P^3$ determined by $|\theta|$ is not
birational. Then it is $2:1$ onto an elliptic normal
quartic curve and $C$ is bielliptic. Then $C$ depends on at most 16 moduli (see 
Lemma \ref {l:dim-loci}), in contradiction with \eqref {eq:m}. 

We have thus proved that $|\theta|$ determines a birational morphism $C \to \P^3$ onto a
degree $8$ curve $\Gamma$ in $\P^3$. Since the maximal geometric genus of a curve of degree
$8$ in $\P^3$ is $9$, then $\Gamma$ is a Castelnuovo curve, so it is a complete intersection of type $(2,4)$, proving the assertion.\end{proof}

\paragraph{$g_1=4$ and $k=2$}
\label{p:g4.k2}

Next we consider $\S _{13} ^{\frac 1 2, 4}$. If $M$ is an irreducible component of
$\S _{13} ^{\frac 1 2, 4}$ whose general element $(C, \theta)$ has $h^0(\theta)$ odd, 
then by Theorem~\ref{t:harris} one has 
\begin{equation}\label{eq:lop}
\textstyle
\dim (M)
\geq \dim(\M_{13}) - \frac {4\cdot 5} 2 
= 26.
\end{equation}
The locus $\mathcal H_{13,6}$ introduced in paragraph~\ref
{p:nonbirat.hyperell} has dimension 25 and it is contained in $\S
_{13} ^{\frac 1 2, 4}$; its general element $(C,\theta)$ has
$h^0(\theta)=7$ odd, so we conclude by
Theorem~\ref{t:harris}
that $\mathcal H_{13,6}$ cannot fill up an irreducible component
of $\S _{13} ^{\frac 1 2, 4}$. 

There are two more hyperelliptic loci of dimension 25 contained in $\S _{13} ^{\frac 1 2, 4}$, namely $\mathcal H_{13,5}^2$ and $\mathcal H_{13,4}^4$, see \ref {p:nonbirat.hyperell}. For the same reasons as above, $\mathcal H_{13,4}^4$ cannot fill up an irreducible component
of $\S _{13} ^{\frac 1 2, 4}$.

\begin{theorem}\label{prop:comp}
$\S _{13} ^{\frac 1 2, 4}$ has three irreducible components, namely:\\
\begin{inparaenum}
\item [(a)] one whose general elements  correspond to
complete intersections of type $(2,2,3)$ in $\P^4$ (see Proposition
\ref {prop:compint}), which contains the loci $\mathcal H_{13,6}$ and
$\mathcal H_{13,4}^4$;\\
\item [(b)] $\mathcal H_{13,5}^2$;\\
\item [(c)] $\mathcal E_{13,2}$ (see paragraph \ref {p:nonbirat.biell}).
\end{inparaenum}
\end{theorem}

\begin{proof}
Let $(C,\theta)$ be general in some irreducible component $M$ of $\S _{13} ^{\frac 1 2, 4}$. By Clifford's Theorem, we have $h^0(\theta)\leq 7$, and, as we saw, the case $h^0(\theta)=7$ cannot occur.

Suppose first that $h^0(\theta)>5$; then the only possibility is
$h^0(\theta)=6$. The linear system $|\theta|$ cannot be birational by
Castelnuovo's bound, so it determines a $2:1$ morphism of $C$ to a
non-degenerate, linearly normal curve $\Gamma$ of degree $d\leq 6$ in
$\P^5$. If $d=6$ then $\Gamma$ has genus 1 and $C$ is bi-elliptic. In
this way we find the locus $\mathcal E_{13,2}\subseteq \S _{13}
^{\frac 1 2, 5}\subseteq \S _{13} ^{\frac 1 2, 4}$. This is an
irreducible component of $\S _{13} ^{\frac 1 2, 4}$, because it cannot
be contained in an irreducible component of $\S _{13} ^{\frac 1 2, 4}$
whose general element $(C,\theta)$ has $h^0(\theta)=5$, since the
parity of $h^0(\theta)$ has to be preserved in a component. If $d=5$,
then $\Gamma$ is a rational normal curve, $C$ is hyperelliptic and we
find the locus $\mathcal H_{13,5}^2$. This is a component of $\S _{13}
^{\frac 1 2, 4}$ because it cannot be contained in an irreducible
component of $\S _{13} ^{\frac 1 2, 4}$ whose general element
$(C,\theta)$ has $h^0(\theta)=5$, and neither can it be contained in
$\mathcal E_{13,2}$ which has smaller dimension.

Next we turn to the case in which the general element $(C,\theta)$ in $M$ is such that $C$ is not hyperelliptic and 
$|\theta|$ is a $g^4_{12}$. Let $b$ be the
number of its base points, so that $|\theta|$ induces a $g^4_{12-b}$.
By Clifford's Theorem, we have $b<4$.
If $b>0$ then $|\theta|$ cannot be birational by Castelnuovo's bound.
If $b=3$,  the $g^4_9$ must
give a birational map, and this is impossible.
If $b=2$, we have a $g^4_{10}$, which is  $2:1$
onto an elliptic normal quintic.
Then $C$ is bi-elliptic. However bielliptic curves depend on 24 moduli (see
Lemma \ref {l:dim-loci}), whereas $\dim(M)\geq 26$ by \eqref {eq:lop}, so this is against the generality of 
$(C,\theta)$ in $M$. Finally, if $b=1$, $|\theta|$ is a $g^4_{11}$ hence it 
gives a birational map, which is impossible.

So  $|\theta|$ is a base point free $g^4_{12}$.
First we discuss the case in which this series is not birational. In this case we have two possibilities:\\
\begin{inparaenum}
\item [(a)] either the $g^4_{12}$ determines a $3:1$ morphism 
onto a rational normal quartic,\\
\item [(b)] or the $g^4_{12}$ determines a $2:1$ morphism 
onto a degree $6$ curve.
\end{inparaenum}

We first show by contradiction that case (a) cannot happen. Indeed,
in this case $C$ is trigonal, we denote by $\mathfrak g$ the $g^1_3$ on $C$, and we have $K_C=8\mathfrak g$. 
The canonical model of $C$ in $\P^{12}$ sits on a
rational normal scroll $S$ of degree $11$, which is easily seen to be
smooth.
On $S$ we have $C \lineq 3H-9F$ where
$H$ is  the hyperplane class on $S$, and $F$ the class of a
ruling.
Now $\restr F C=\mathfrak g$ and $\restr H C=K_C=8\mathfrak g$.
Hence $S$ should have a hyperplane section consisting of
$8$ rulings plus a curve $E$ of degree $3$ such that $E\cdot C=0$. The 
curve $E$ is  irreducible (because $E\cdot F=1$ and it
cannot contain fibres), and one has $E^2=-5$. It follows that
$S$ is an $\F_5$, and $H\lineq L+3F$, with $L\lineq
E+5F$. Thus $C \lineq 3L$. The linear system
$|3L|$ on $\F_5$ has dimension $33$, and $\Aut(\F_5)$ has dimension $10$, so the
curve $C$ has at most $23$ moduli, and therefore cannot be a general
element of $M$ by \eqref {eq:lop}. 
This proves that (a) cannot happen.

Next we show that case (b) also cannot happen. 
In that case in fact the image curve of $C$ via the 
$g^4_{12}$ is linearly normal of degree 6 in $\P^4$, so it has genus 2.
By Lemma~\ref{l:dim-loci}, genus $13$ double covers of genus $2$
curves have $23$ moduli, so also this possibility is in contradiction
with the generality of $C$ by \eqref {eq:lop}. 

Therefore the only remaining possibility is that $|\theta|$ is a
birational $g^4 _{12}$. Let $\Gamma$ be the image of $C$ via the morphism determined
by $|\theta|$. 
One has
\[
h^0(\I_\Gamma(2)) \geq h_4(2)-h^0(\omega_C)=15-13=2.
\]
Assume for a moment that
$h^0(\I_\Gamma(2))=2$. Then, since $3\theta = K_C+\theta$ is
non-special, one has
 $h^0(3\theta) = 24$ by Riemann--Roch Theorem, hence
\[
h^0(\I_\Gamma(3)) \geq h_4(3)-h^0(3\theta)=35-24=11,
\]
and therefore there exists at least one cubic containing $\Gamma$, not
a combination of the
 quadrics containing $\Gamma$.
For degree reasons we may  conclude that $\Gamma$ is a complete
intersection of type $(2,2,3)$ as asserted.

We are thus left to prove that $h^0(\I_\Gamma(2))=2$. Assume by
contradiction that $h^0(\I_\Gamma(2))>2$. Since $\deg(\Gamma)>8$,  the quadrics containing $\Gamma$
must have base locus a non-degenerate surface, and the only
possibility is that this is a cubic scroll $S$
and then $h^0(\I_\Gamma(2)) = h^0(\I_S(2))=3$.
Now, with the same computations as above, we see that $h^0(\I_\Gamma(5))\geq 78$, whereas $h^0(\I_S(5))= 75$, so there are quintic hypersurfaces containing $\Gamma$ but not $S$. Let 
$H$ be the hyperplane section of $S$ and set $D=5H-\Gamma$. We have $\dim(|D|)\geq 2$ and $\deg(D)=3$. Given this, the possibilities for $D$ are the following:\\
\begin{inparaenum}
\item [(i)] $D\lineq H$;\\
\item [(ii)] $S$ is not a cone and $D$ consists of 3 rulings;\\
\item [(iii)] $S$ is not a cone and $D$ consists of a conic plus the line directix $E$ of $S$.\\
\end{inparaenum}
We will see that neither one of these cases is possible.

In case (i), we have $\Gamma\lineq 4H$. Then the arithmetic genus of $\Gamma$ is 15, hence $\Gamma$ has two nodes or equivalent singularities because the geometric genus is 13. The adjoint system is $4H+K_S\lineq 2H+F$, where $F$ is a ruling of $S$, hence it is possible that the pull-back on $C$ of a
hyperplane section is semi-canonical if the two nodes of $\Gamma$ lie on a ruling of $S$.  Let us count the number of moduli for such curves $\Gamma$. We have $h^0(4H)=35$. For each ruling we have $2$ parameters for the choice of the nodes of $\Gamma$ on that ruling, hence $3$ parameters in total for the position of
the nodes. Given the nodes we have a linear system of dimension $34-6=28$ (and not larger) of curves in $|4H|$ with the chosen nodes. Finally the automorphism group of $S$ has at least
dimension $6$, so  the number of moduli is not larger than $3+28-6=25$. By \eqref {eq:lop} this contradicts the generality of $(C,\theta)$ in $M$. 

In case (ii) we have
\[
  \Gamma\lineq 5H-3F=5E+7F
  \quad \text{and}
  \quad
  K_S+\Gamma \lineq 3E+4F,
\]
with $E$ the negative curve on $S$.
Then  $\Gamma$ has arithmetic genus 14, so it has only one node (or one cusp).  It is possible that the pull-back on $C$ of a
hyperplane section is semi-canonical if the double point of $\Gamma$ lies on $E$.  Again, let us count the number of moduli for such curves $\Gamma$. We have $h^0(5E+7F)=33$. The choice of the double point of $\Gamma$ depends on only 1 parameter (a point on $E$). Given the double point we have a linear system of dimension $32-3=29$ (and not larger) of curves in $|5E+7F|$ with the chosen double point. Finally the automorphism group of $S$ has 
dimension $6$, so  the number of moduli is $1+29-6=24$. By \eqref {eq:lop} this again contradicts the generality of $(C,\theta)$ in $M$. 

Finally in case (iii) we have $\Gamma\lineq 3(E+3F)$, whose arithmetic genus is 13. So $\Gamma=C$ is smooth. We have
\[
  K_S+C \lineq E+6F.
\]
Assume that the hyperplane series on $C$ is semi-canonical. This means that 
\[
\restr{E+6F} C \lineq \restr{2E+4F}C \quad \text{hence}\quad  \restr E C\lineq \restr {2F} C,
\]
which is impossible: indeed $|2F|$ cuts out on $C$ a complete $g^2_6$
composed with the $g^1_3$ cut out by $|F|$. If $\restr E C\lineq
\restr {2F} C$, then the six points cut out by $E$ on $C$ would also
belong to the union of two rulings, which is not
possible.
\end{proof}

\paragraph{$g_1=5$ and $k=2$}
\label{p:g1=5:k=2}

This case is more complicated than the previous ones and we do not
have complete results. We have here $g=17$. If $M$ is an irreducible
component of $\S _{17} ^{\frac 12,5}$ with general element
$(C,\theta)$ such that $h^0(\theta)$ is even then, as usual, we have 
\begin{equation}	\label{eq:pap}
\textstyle
\dim (M)\geq  48 - \frac {5\cdot 6} 2 = 33.
\end{equation}
There is the  \emph{main component}
$(\S _{17} ^{\frac 1 2,5})^{\mathrm{main}}$,
that of complete intersections $(2,2,2,2)$,
that has 
dimension
\[
  \dim \bigl(
  (\S _{17} ^5)^{\mathrm{main}}
  \bigr)
  = \dim \bigl(
  \Gr(4,H^0(\O_{\P^5}(2)))
  \bigr)
  - \dim \bigl(
  \PGL(6)
  \bigr)
  = 4 \cdot 17 - 35 = 33
\]
equal to the expected one since the general element $(C,\theta)$ has
$h^0(\theta)=6$.

Let us look for non-birational loci. The interesting hyperelliptic loci we found in Section \ref {p:nonbirat.hyperell} are $\mathcal H_{17,8}$, $\mathcal H_{17,7}^2$, $\mathcal H_{17,6}^4$, $\mathcal H_{17,5}^6$, all of dimension $33$. If $(C,\theta)$ is in $\mathcal H_{17,8}$, one has $h^0(\theta)=9$ which has different parity than $6$. So if $M$ is a component of $\S _{17} ^{\frac 12,5}$ containing $\mathcal H_{17,8}$ we have the estimate  
\begin{equation}\label{eq:pap2}
\textstyle
\dim(M)\geq  48 - \frac {6\cdot 7} 2 = 27.
\end{equation}
The same estimate holds for a component $M$ of $\S _{17} ^{\frac 12,5}$  containing $\mathcal H_{17,6}^4$. By contrast, if $(C,\theta)$ is in $\mathcal H_{17,7}^2$ or in  $\mathcal H_{17,5}^6$ then $h^0(\theta)$ is even, and for a component of $\S _{17} ^{\frac 12,5}$ containing these loci the estimate \eqref {eq:pap} holds. 

Look next at the bielliptic locus $\mathcal E_{17,2}\subseteq \mathcal S_{17}^{\frac 12, 7}\subseteq \S _{17} ^{\frac 12,5}$ (see Section \ref {p:nonbirat.biell}). We have $\dim(\mathcal E_{17,2})=32$. If $M$ is an irreducible component of $\S _{17} ^{\frac 12,5}$ containing $\mathcal E_{17,2}$, we have the estimate \eqref {eq:pap}. So $\mathcal E_{17,2}$ has at least codimension 1 in an irreducible component of $\S _{17} ^{\frac 12,5}$ containing it.

Finally, look at the locus $\mathcal D_{17,2}\subseteq S_{17}^{\frac 12, 6}\subseteq \S _{17} ^{\frac 12,5}$ (see Section \ref {p:nonbirat.g2}). We have $\dim(\mathcal D_{17,2})=33$. If $M$ is an irreducible component of $\S _{17} ^{\frac 12,5}$ containing $\mathcal D_{17,2}$, we have the estimate \eqref {eq:pap2}. 

\begin{proposition}\label{prop:g1=5} Let $M$ be an irreducible component of $\S _{17} ^{\frac 12,5}$ and let $(C,\theta)$ be its general element. Suppose that $C$ has no $g^1_4$. Then $h^0(\theta)=6$, and $|\theta|$ is base point free and birational. \end{proposition}

\begin{proof}
Let $M$ be an irreducible component of $\S _{17} ^{\frac 12,5}$ and let $(C,\theta)$ be its general element. By the hypothesis, $C$ is not hyperelliptic. Hence $h^0(\theta)< 9$ by Clifford's Theorem. Moreover, if $h^0(\theta)>6$, then $|\theta|$ is not birational, by Castelnuovo's bound. 

If $h^0(\theta)=7$, since $|\theta|$ is not birational, it must determine a 
$2:1$ morphism of $C$ onto a non-degenerate linearly normal curve $\Gamma$ of degree $d\leq 8$ in $\P^6$. Then $\Gamma$ has genus at most 2, hence $C$ has a $g^1_4$, a contradiction. 

Similarly, if $h^0(\theta)=8$,  $|\theta|$  determines a $2:1$ morphism of $C$ onto a non-degenerate linearly normal curve $\Gamma$ of degree $d\leq 8$ in $\P^7$. Then $\Gamma$ has genus at most 1, so that $C$ has a $g^1_4$ a contradiction again. 

So we may assume $h^0(\theta)=6$ and the estimate \eqref {eq:pap} holds. We claim that $|\theta|$ is birational. 
Suppose this is not the case. Let $h\geq 2$ be the degree of the map determined by $|\theta|$ onto a non-degenerate curve $\Gamma$ of degree $d$ in $\P^5$. We have 
\[
5\leq d\leq \frac {16}h \quad \text{hence} \quad 5h\leq 16.
\]
Then either $h=3$, $d=5$, or $h=2$ and one has $d>5$ since $C$ is not
hyperelliptic. In the former case $C$ would be trigonal, a
contradiction. In the latter cases, by Lemma \ref {l:dim-loci}, $C$
would depend on at most 32 moduli, contradicting the estimate
\eqref {eq:pap}. This proves that $|\theta|$ is birational.

Finally we prove that  $|\theta|$ is base point free.
Let $b$ be the number of base points of $|\theta|$, and consider the
complete, base-point-free, special  $g^5_{16-b}$ induced by $|\theta|$.
We have $b\leq 1$ because, by
Castelnuovo's bound, a curve in $\P^5$ of degree
$d\leq 14$ has genus at most $15$. So let us discuss the case $b=1$. 

If $b=1$ we have a birational, base point free $g^5_{15}$, which maps
$C$ to a curve 
$\Gamma$ of degree 15 in $\P^5$. We claim that $\Gamma$  is contained in at least $6$ linearly independent quadrics.
Indeed we have
\[
  h^0(\O_{\Gamma}(2)) \leq h^0 \bigl(\omega_C(-2p)
  \bigr)=15
\]
where $p$ is the base point of $|\theta|$
(the last equality is because $C$ is not hyperelliptic), hence
\[
h^0(\I_{\Gamma}(2)) \geq h_5(2)-h^0(\O_{\Gamma}(2)) 
\geq 21-15=6.
\]
If the quadrics containing $\Gamma$ have base locus a curve, then
there exists a line $\ell$ such that $\Gamma+\ell$ is a
complete intersection of type $(2,2,2,2)$. Then $\Gamma$, and its
normalization $C$, have strictly less number of moduli than 33, the number of moduli of complete
intersections of type $(2,2,2,2)$, in contradiction with the estimate \eqref {eq:pap}.

Otherwise the base locus of the quadrics containing $\Gamma$ is a rational normal quartic scroll.
In fact this base locus cannot be a threefold because the maximum number of quadrics containing an irreducible, non-degenerate threefold in $\P^5$ is 3, and it is achieved by the Segre variety $\P^1\times \P^2$ of degree 3. So the base locus in question is a surface, and it is a minimal degree surface $S$, which in fact has the property that $h^0(\mathcal I_{S}(2))=6$. On the other hand $S$ cannot be the Veronese surface because there are no curves of odd degree on it. 

Finally we discuss the case in which $\Gamma$ is contained in a rational normal quartic scroll $S$. If $S$ is a cone with vertex $p$, let $m$ be the number of points in which the rulings intersect $\Gamma$ off $p$, so that $C$ has a $g^1_m$. By intersecting $\Gamma$ with a general hyperplane through $p$ we see that $4m\leq 15$, which implies $m\leq 3$, a contradiction. 

Next we assume $S$ to be smooth. By Riemann--Roch we have
$h^0(\mathcal O_\Gamma(5))\leq 59$ and $h^0(\mathcal O_S(5))=66$
(whether $S$ be an $\F_0$ or an $\F_2$), so 
\[
h^0(\mathcal I_{\Gamma}(5))\geq 252-59>252-66=h^0(\mathcal I_{S}(5)).
\]
Hence there are quintic hypersurfaces containing $\Gamma$ that do not contain $S$. They cut out on $S$ a linear system $\mathcal L$ of dimension at least 6 of quintic curves $D$. The curves in $\mathcal L$ cut out on a general hyperplane section $H$ of $S$ a linear series $g^r_5$  with $r\leq 5$. Hence there are curves in $\mathcal L$ containing $H$. So $D\lineq H+R$, where $R$ is a line on $S$, and we have $\Gamma \lineq 
5H-D\lineq 4H-R$. If $F$ is a ruling of $S$, then $F\cdot\Gamma\leq
4$, hence $C$ has a $g^1_4$, a contradiction.
\end{proof}\medskip

Unfortunately the discussion whether there is only one component as in the statement of Proposition \ref {prop:g1=5}, coinciding  with $ (\S _{17} ^5)^{\mathrm{main}}$, is quite intricate and we will not dwell on this here.

\subsection{The case $k=3$}\label{sec:k3}

\paragraph{$g_1=3$ and $k=3$}
Here we analyze the possibilites for curves in $\S _{19} ^{\frac 1 3, 3}$,
\ie pairs $(C,\theta)$ with $C$ of genus 19, $h^0(\theta)\geq 4$, and $3\theta \lineq K_C$,
hence $\deg(\theta) = 12$.
We do not have an expected dimension for $\S _{19} ^{\frac 1 3, 3}$, but we know that
\[
\dim \bigl(\im (c_{19}^3) \bigr)
= \dim (\KC _{19} ^3) - 4
= 34
\]
(see  \ref {par:fibre}).

There are certainly irreducible loci in $\S _{19} ^{\frac 1 3, 3}$,
whose general element $(C,\theta)$ 
is such that $|\theta|$ is not birational. For example, consider $\mathcal H_{19,6}$ which has dimension 37 and its general element 
$(C,\theta)$ is such that $h^0(\theta)=7$. However we can prove the following:

\begin{proposition}\label{prop:no} If $(C,\theta)\in \S _{19} ^{\frac 1 3, 3}$ and $C$ has no $g^1_6$, then $|\theta|$ is birational. 
\end{proposition}

\medskip\noindent
(By way of comparison, a general curve complete intersection of type
$(3,4)$ has Clifford index $6$ by Lemma~\ref{l:cliff}, hence has no
$g^1_7$). 

\begin{proof} Suppose  $(C,\theta)$ is such that $|\theta|$ is not birational. Since $C$ is not hyperelliptic, by Clifford's Theorem we have $h^0(\theta)<7$, which leaves the possibilities $4\leq h^0(\theta)\leq 6$.

Suppose that $h^0(\theta)=4$. Let $b$ be the number of base points of $|\theta|$. Since $C$ is non-hyperelliptic, again by Clifford's Theorem we have $b\leq 5$. If $b=5$, $|\theta|$ determines a $g^3_7$, which should be birational, a contradiction.
If $b=4$ we obtain a $g^3_8$. It may give a double cover of a curve of degree 4, which has genus at most 1.
Then $C$ has a $g^1_4$, a contradiction.
If $b=3$ we have a $g^3_9$. This gives a $3:1$ map to a rational normal cubic curve, hence $C$ is trigonal, a contradiction.
If $b=2$ we have a $g^3_{10}$, which should give a double cover of a
quintic; a quintic in $\P^3$ has genus at most 2, so $C$ has a $g^1_4$, a contradiction.
If $b=1$ we have a $g^3_{11}$, necessarily birational, a contradiction.
If $|\theta|$ is a non-birational, base point free $g^3_{12}$, it
gives a double cover of a sextic curve,
or a triple cover of a quartic curve. In both cases we see that $C$ has a $g^1_6$, a contradiction again.
So the case $h^0(\theta)=4$ is ruled out.

If $h^0(\theta)=5$, with the same notation as above, by Clifford's Theorem we have $b\leq 3$.
If $b=3$, $|\theta|$ determines a $g^4_9$ which should be birational, a contradiction. 
If $b=2$, we have a $g^4_{10}$ which gives a $2:1$ map to a quintic, which has genus 1, so $C$ has a $g^1_4$, a contradiction. If $b=1$, we have a $g^4_{11}$ which should be birational, a contradiction. 
If $b=0$, then $|\theta|$ is a $g^4_{12}$ which either determines a
$3:1$ morphism of $C$ to a rational normal quartic, or a $2:1$
morphism of $C$ to a genus 2 sextic curve, both cases leading to a
contradiction.
So also the case $h^0(\theta)=5$ is ruled out.

The case $h^0(\theta)=6$ can be ruled out with similar arguments, we
leave the details to the reader.
\end{proof}\medskip

Actually we can be more precise:

\begin{proposition}\label{prop:bir} Let $(C,\theta)\in \S _{19} ^{\frac 1 3, 3}$ be such that $C$ has no $g^1_6$. Then $C$ is a complete intersection of type $(3,4)$ in $\P^3$. 
\end{proposition}

\begin{proof} By Proposition \ref {prop:no}, we have that $|\theta|$ is birational. 
By Castelnuovo's bound we must have $h^0(\theta)=4$.
Let $\Gamma\subset \P^3$ be the image of $C$ via the morphism determined by $|\theta|$. One has $\deg(\Gamma)\leq 12$. We claim that $\Gamma$ does not lie on a quadric. In fact, suppose that $C$ lies on a smooth quadric $Q=\P^1\times \P^1$. Then it is a curve of type $(a,b)$ with $a+b=\deg(\Gamma)\leq 12$. This implies that $\min\{a,b\}\leq 6$, so that $C$ has a $g^1_6$ a contradiction. A similar argument shows that $\Gamma$ does not lie on a quadric cone.

Suppose next that $|\theta|$ has $b>0$ base points.
By Castelnuovo's bound, one has $b=1$, so there is only one base point $p$.
We have
$h^0(\I_{\Gamma}(3)) \geq h_3(3) -h^0(\O_{\Gamma}(3))$
and
\[
  h^0(\O_{\Gamma}(3)) \leq h^0(\O_C(3))
  = h^0(3\theta-3 p) < h^0(K_C) =19.
\]
Then $h^0(\I_{\Gamma}(3)) \geq 2$.
Hence $\Gamma$ is contained in two distinct cubics $X$ and $Y$, with no common component. 
This is a contradiction because we would have $11=\deg(\Gamma)\leq 9$. This proves that $|\theta|$ is base point free.

Then we have
\[
h^0(\mathcal O_\Gamma(3))\leq h^0(\mathcal O_C(3\theta))=h^0(K_C)=19.
\]
Hence 
\[
h^0(\mathcal I_\Gamma(3))\geq h_3(3)-19=1,
\]
so there is an irreducible cubic $X$ containing $\Gamma$. Similarly, we see that $h^0(\mathcal I_\Gamma(4))\geq 5$, so there is some quartic surface $Y$ containing $C$ which does not contain the cubic $X$. Then $\Gamma$ is the complete intersection of $X$ and $Y$, as desired.  \end{proof}\medskip

\begin{remark}
\label{conj:gen2} 
The hypothesis that $C$ has no $g^1_6$ in Proposition \ref{prop:bir}
is to exclude that $|\theta|$ is non-birational and that, if
$|\theta|$ is birational, the curve $\Gamma$ lies on a quadric. On the
other hand we do know that there are points $(C,\theta)\in \S _{19}
^{\frac 1 3, 3}$ corresponding to curves on a quadric. One example is
the locus $\S$ considered in paragraph~\ref {p:g3.quadric} although,
if we believe in Conjecture \ref {conj:gen}, it should not give a
new irreducible component of $\S _{19} ^{\frac 1 3, 3}$.

Another example of points $(C,\theta)\in \S _{19} ^{\frac 1 3, 3}$, with  $|\theta|$ birational, corresponding to curves on a quadric, is the following. Consider a smooth quadric $S=\P^1\times \P^1$ in $\P^3$. Consider on $S$ curves $\Gamma$ of type $(5,6)$ with one single double point $q$ (a node or a simple cusp), and such that the $5$-secant ruling through $q$ intersects $\Gamma$, off $q$, in a divisor of the type $3p$. Let $\nu: C\to \Gamma$ be the normalization. Since the arithmetic genus of $\Gamma$ is 20, then $C$ has genus 19. 
Consider the line bundle $\theta=\nu^*(\O_\Gamma(1))\otimes \O_C(p)$, where we abuse notation and see $p$ as a point of $C$. One checks that $3\theta=K_C$, so $(C,\theta)\in \S _{19} ^{\frac 1 3, 3}$.
Here $|\theta|$ has the base point $p$.  One sees that the number of moduli on which the construction of these curves depend is 32. As in \ref {conj:gen}, one may conjecture that these curves are limits of complete intersections. 
\end{remark}

\paragraph{$g_1=4$ and $k=3$}
\label{p:g4.k3}
Here we analyze the possibilites for elements $(C,\theta)$ in $\S _{28} ^{\frac 1 3, 4}$,
 with $C$ of geometric genus $28$, $h^0(\theta)\geq 5$, and
$3\theta \lineq K_C$,
hence $\deg(\theta) = 18$.
We focus on the case when $|\theta|$ defines a
birational map. Then we must have $5\leq h^0(\theta)\leq 6$ by Castelnuovo's
bound. 

\begin{proposition}\label{prop:g_1=4,k=3} Assume $(C,\theta)$ in $\S _{28} ^{\frac 1 3, 4}$ is such that $|\theta|$ is base point free,  defines a
birational map and $C$ has no $g^1_8$. Then $C$ is a complete intersection of type $(2,3,3)$ in $\P^4$.
\end{proposition}

\medskip\noindent
(By way of comparison, a general curve complete intersection of type
$(3,4)$ has Clifford index $10$ by Lemma~\ref{l:cliff}, hence has no
$g^1_{11}$). 

\begin{proof} First we assume $h^0(\theta)= 6$. The maximum genus of non-degenerate curves of degree 18 in $\P^5$ is 28, so $|\theta|$ embeds $C$ in $\P^5$ as a smooth Castelnuovo curve. These curves are classified (see
\cite[p.~122--123]{acgh}), and they are of two types, namely:\\
\begin{inparaenum}[(a)]
\item $C\subset \P^5$ is the $2$-Veronese image of a smooth plane curve of degree $9$;\\ 
\item $C$ is a smooth element of $|5H-2F|$ on a rational normal
scroll $S$, with $H$ the hyperplane section and $F$ a ruling of $S$.
\end{inparaenum}

In either case $C$ has a $g^1_d$ with $d\leq 8$, so we can exclude these cases.

Assume next that $h^0(\theta)= 5$ and let $\Gamma$ be the image of $C$ via $|\theta|$. 
We have $h^0(\O_\Gamma(2))\leq h^0(2\theta)=h^1(\theta)=14$ by Riemann--Roch, hence
$h^0(\I_\Gamma(2))\geq 1$, thus $C$ sits on an irreducible quadric $Q$.
Next $h^0(\O_\Gamma(3))\leq h^0(K_C)=28$, hence
$h^0(\I_\Gamma(3)) \geq 7$, thus $C$ sits on at least two cubics $X$ and
$X'$ that do not contain $Q$. If $Q, X$ and $X'$ intersect in a curve, by degree reasons this curve is $\Gamma$, which is a complete intersection of type $(2,3,3)$ as desired. 

So we have to discuss the case in which $Q, X$ and $X'$ intersect along an irreducible surface $S$ containing $\Gamma$. Since $X'$ is independent of $Q$ and $X$, we see that $\deg(S)\leq 5$. 

Assume first $\deg(S)= 5$. In this case $S$ is the complete intersection of a quadric and a cubic off a plane $\Pi$. This is an arithmetically Cohen--Macaulay surface whose Hilbert--Burch matrix is of the type
\begin{equation}\label{eq:mat}
\begin{pmatrix} 
    1 & 1 & 2 \\
    1 & 1 & 2 
  \end{pmatrix}
\end{equation}
(the numbers at the entries stand for the degrees of the corresponding
forms) and its general hyperplane section $D$ is a curve with
arithmetic genus 2. By adjunction, the hyperplanes containing $\Pi$
cut out on $S$, off the intersection scheme of $S$ with $\Pi$, a
pencil $\mathcal P$ which is adjoint to the hyperplane section of $S$,
so the curves in $\mathcal P$ have degree 2. If the general curve in
$\mathcal P$ is reducible, then $S$ is a scroll. Otherwise the general
curve in $\mathcal P$ is an irreducible conic and $S$ is rational.

We compute $h^0(\O_\Gamma(4))\leq h^0(K_C+\theta)=45$, hence
$h^0(\mathcal I_\Gamma(4))\geq 25$. From the matrix \eqref {eq:mat} we
deduce that $h^0(\mathcal I_S(4))= 23$, so there are quartics
containing $\Gamma$ and not $S$. If $S$ is a scroll, its genus
is the geometric genus of $D$, which is at most 2.
This   implies that  $C$ has a $g^1_8$, so this case is not possible. 
If the general conic of $\mathcal P$ is irreducible, the pull-back on
$C$ of the linear series cut out by these conics on $\Gamma$ is a
$g^1_d$, with $d\leq 8$, so that we can rule out this case too.

Assume next $\deg(S)=4$. Since $\Gamma$ is linearly normal, it cannot sit on a non-linearly normal surface, so $S$ is the complete intersection of two quadrics and it has only isolated singularities. So $S$ can be either a cone over an elliptic normal quartic curve in $\P^3$ or a del Pezzo surface.

If $S$ is a cone, with vertex $p$, let $m$ be the number of intersection points of $\Gamma$ with a ruling, off $p$. By taking a general hyperplane section of $\Gamma$ through $p$, we see that $4m\leq 18$ so that $m\leq 4$, which implies that $C$ has a $g^1_8$, and we can rule out this case.

Suppose $S$ is a del Pezzo surface. Thus $S$ is Gorenstein and
$K_S=\O_S(-1)$. Moreover $S$ is represented on $\P^2$ by a linear
system of cubics passing simply through five points $p_1,\ldots, p_5$,
which can be proper or infinitely near.
With computations usual by now, we find
that $h^0(\mathcal I_\Gamma(6))\geq 129$, whereas $h^0(\mathcal
I_S(6))=125$, so there are sextic hypersurfaces containing $\Gamma$
and not $S$. We let $\nu: \tilde S\to S$ be the minimal
desingularization of $S$ and, abusing notation, we denote by $\Gamma$
its proper transform on $\tilde S$. We set $H=\nu^*(\O_S(1))$, and
notice that $K_{\tilde S}=-H$. We have $6H\lineq \Gamma+D$, with
$\dim(|D|)\geq 3$. Moreover, $\Gamma+K_{\tilde S}\lineq 5H-D$. Since
the pull-back of $H$ on $C$ is $\theta$ and $3\theta=K_C$, we have
that the pull-back to $C$ of the divisors in $|2H-D|$ verifying the
adjoint conditions with respect to $\Gamma$ is the $0$ divisor (off
the pull-back of singular points). Thus we have $2H\sim D+E$, where
$E$ is effective and $H\cdot E=2$, so the curves in $|E|$ correspond
to curves of degree 2 on $S$.  There are two possibilities for the
linear system $|E|$:\\
\begin{inparaenum}
\item [(a)] $|E|$ has dimension 1 and it corresponds to a linear
system of conics on $S$ plus possibly a fixed part consisting of
$(-2)$-curves: we may then suppose that the movable part of $|E|$ is
represented in $\P^2$ by the pencil of lines through one of the points
$p_1,\ldots, p_5$;\\
\item [(b)] $h^0(E)=1$ and $E=E_1+E_2+Z$, where $E_1,E_2$ are two
  $(-1)$-curves, either distinct or not, and $Z$ is a (possibly
  empty) union of  $(-2)$-curves.
\end{inparaenum}

In case (a),  $|D|$ is represented on $\P^2$ by a linear system of curves of degree 5 with  multiplicity 1 at one of the points $p_1,\ldots, p_5$ and $2$ at the others. Since $\Gamma\lineq 6H-D$,  then $\Gamma$ is represented on $\P^2$ by a curve of degree $13$ with multiplicity at least $5$ at one of the points $p_1,\ldots, p_5$,  and this implies that $C$ has a $g^1_8$, a contradiction. 

In case (b) we can either have $E_1=E_2$ or $E_1\neq E_2$. Assume  that $E_1=E_2$. Then we may suppose that $E_1$ is the $(-1)$-curve corresponding to one of the points $p_1,\ldots, p_5$.  Then $|D|$ is represented on $\P^2$ by a linear system of curves of degree 6 with a point of multiplicity 4 and four points of multiplicity 2.  Then $\Gamma$ is represented on $\P^2$ by a curve  of degree 12 with a point of multiplicity 2 and 4 points of multiplicity 4. Again this implies that $C$ has a  $g^1_8$, a contradiction. The case in which  $E_1$ and $E_2$ are distinct is similar, and also leads to an impossibility, we leave the details to the reader.

Finally we assume $\deg(S)=3$, so that $S$ is a rational normal scroll. 
If $S$ is smooth, then  $S$ is isomorphic to  $\P^2$ blown-up at a point $p$. Suppose that $\Gamma$ is represented on  $\P^2$ by a curve of degree $d$ with a point of multiplicity $m$ at $p$. We have
\[
18= \deg(\Gamma)=2d-m
\]
hence $d-m\leq 18-d$. Since $C$ has no $g^1_8$ we must have $d\geq 10$. But then $d-m\leq 8$. Since $C$ has a $g^1_{d-m}$ we can rule out this case.

If $S$ is a cone with vertex $p$, we let $m$ be the number of intersection points of the rulings with $\Gamma$ off $p$, so that $C$ has a $g^1_m$. By intersecting $\Gamma$ with a general hyperplane through $p$, we see that $3m\leq 18$, so $m\leq 6$, and we can rule out this case too.\end{proof}\medskip

\begin{remark}
\label{conj:genal} 
The hypothesis that $|\theta|$ is base point free and $C$ has no $g^1_8$ in Proposition \ref {prop:g_1=4,k=3} helps to rule out several cases which in fact can occur. We point out here some of these possibilities.\\
\begin{inparaenum}
\item [(a)] As seen in the proof of Proposition \ref {prop:g_1=4,k=3}, the 2-Veronese image $C$ in $\P^5$ of a smooth plane curve of degree 6, is such that $(C, \O_C(1))\in \S_{28}^{\frac 13,4}$. This gives a locus in $\S_{28}^{\frac 13,4}$ of dimension 46.\\
\item [(b)] Again, we saw in the proof of Proposition \ref {prop:g_1=4,k=3} that the smooth Castelnuovo curves $C$ in $|5H-2F|$ on a rational normal
scroll $S$ in $\P^5$, with $H$ the hyperplane section and $F$ a ruling of $S$ are such that
$(C, \O_C(1))\in \S_{28}^{\frac 13,4}$. This gives a locus in $\S_{28}^{\frac 13,4}$ of dimension 47.\\
\item [(c)] This is an example of elements $(C,\theta)\in
\S_{28}^{\frac 13,4}$, with $h^0(\theta)=5$ and $|\theta|$ with a base
point. Consider a smooth cubic rational normal scroll $S$ in
$\P^4$. We denote by $E$ the $(-1)$-curve on $S$ and $F$ a
ruling. Consider a curve $\Gamma$ in $|10F+7E|$ which has an ordinary
triple point in a point $p$ of $E$, two distinct nodes $p_1,p_2$ along
a ruling $F_0$ of $S$ not passing through $p$, and such that $F_0$
cuts out on $\Gamma$, off $p_1$ and $p_2$, a divisor of type $3q$. The
arithmetic genus of $\Gamma$ is 33, so its geometric genus is 28. If
$\nu: C\to \Gamma$ is the normalization, one sees that
$\theta=\nu^*(\O_\Gamma(1))\otimes \O_C(q)$ (where we abuse notation
and denote by $q$ its pre-image on $C$) is such that $3\theta=K_C$. \\
\item [(d)] This is an example of elements $(C,\theta)\in \S_{28}^{\frac 13,4}$, with $h^0(\theta)=5$ and $|\theta|$ with no base point. Consider a del Pezzo quartic surface $S$ in $\P^4$ which is the image of the linear system of cubics in $\P^2$ with simple base points at general points $p_1,\ldots, p_5$. Consider the linear system $\L$ of plane curves of degree 13 with multiplicity 5 at $p_1$, with multiplicity 4 at $p_2,\ldots, p_5$ and with 4 nodes located along a line $L$ passing through $p_1$. The image of a general curve of $\L$ is a 4-nodal curve of geometric genus 28 and degree 18, with the 4 nodes located along a conic $E$ on $S$. The adjoint linear system of $\L$ consists of curves of degree 10, having multiplicity 4 at $p_1$, 3 at $p_2,\ldots, p_5$ and passing through the 4 nodes. The line $L$ cuts out the zero divisor off the singular points of the general curve $\Gamma\in \L$. So the canonical system is cut out by curves of degree 9 with points of multiplicity 3 at $p_1,\ldots, p_5$, which maps to the  the triple of the hyperplane system of $S$.
\end{inparaenum}
\end{remark}

\begin{closing}


\providecommand{\bysame}{\leavevmode\hbox to3em{\hrulefill}\thinspace}
\providecommand{\og}{``}
\providecommand{\fg}{''}
\providecommand{\smfandname}{and}
\providecommand{\smfedsname}{eds.}
\providecommand{\smfedname}{ed.}
\providecommand{\smfmastersthesisname}{M\'emoire}
\providecommand{\smfphdthesisname}{Th\`ese}


\medskip\noindent
Ciro Ciliberto.
Dipartimento di Matematica.
Università degli Studi di Roma Tor Vergata.
Via della Ricerca Scientifica,
00133 Roma, Italy.
\texttt{cilibert@mat.uniroma2.it}

\medskip\noindent
Thomas Dedieu.
Institut de Mathématiques de Toulouse~; UMR5219.
Université de Toulouse~; CNRS.
UPS IMT, F-31062 Toulouse Cedex 9, France.
\texttt{thomas.dedieu@math.univ-toulouse.fr}

\renewcommand{\thefootnote}{}
\footnotetext
{The first-named author is a member of GNSAGA of INdAM. He
acknowledges the MIUR Excellence Department Project awarded to the
Department of Mathematics, University of Rome Tor Vergata, CUP
E83C18000100006.}
\end{closing}

\end{document}